\documentclass[11pt]{amsart}
\usepackage{color}

\usepackage[all,cmtip]{xy}
\usepackage{amssymb,mathtools}
\usepackage[english]{babel}
\usepackage{mathrsfs}

\newtheorem{theorem}{Theorem}[section]
\newtheorem{lemma}[theorem]{Lemma}

\newtheorem{proposition}[theorem]{Proposition}
\newtheorem{corollary}[theorem]{Corollary}

\theoremstyle{definition}
\newtheorem{definition}[theorem]{Definition}

\newtheorem{example}[theorem]{Example}

\newtheorem{remark}[theorem]{Remark}

\newtheorem{claim}{Claim}

\def\U{\mathbb{U}} %
\def\H{\mathbb{H}}

\DeclareMathOperator{\diam}{diam}
\DeclareMathOperator{\lcm}{lcm}
\DeclareMathOperator{\Iso}{Iso}
\DeclareMathOperator{\Aut}{Aut}
\DeclareMathOperator{\Sym}{Sym}

\def\rng{\mathrm{rng}} %
\def\dom{\mathrm{dom}} %

\newcommand{\sh}{\mbox{-}}
\newcommand{\Q}{\mathbb{Q}}
\newcommand{\Lan}{\mathcal{L}}
\newcommand{\M}{\mathcal{M}}

\newcommand{\K}{\mathcal{K}}

\newcommand{\C}{\mathcal{C}}
\newcommand{\R}{\mathcal{R}}

\usepackage{color}
\newenvironment{revvy}{\color{magenta}}{}
\newenvironment{revdw}{\color{blue}}{}
\newenvironment{revsg}{\color{red}}{}

\def\by{\begin{revvy}}
\def\ey{\end{revvy}}
\def\bd{\begin{revdw}}
\def\ed{\end{revdw}}
\def\bg{\begin{revsg}}
\def\eg{\end{revsg}}

\def\N{\mathbb{N}} %

\def\R{\mathbb{R}}

\def\C{\mathcal{C}}

\def\Q{\mathbb{Q}}

\def\e#1{e}

\begin{document}
\title[Groups with the L\'evy property]{Isometry groups and countable groups with the L\'evy property}

\author[W. Dai]{Wei Dai}
\address{School of Mathematical Sciences and LPMC, Nankai University, Tianjin 300071, P.R. China}
\email{weidai@nankai.edu.cn}

\author[S. Gao]{Su Gao}
\address{School of Mathematical Sciences and LPMC, Nankai University, Tianjin 300071, P.R. China}
\email{sgao@nankai.edu.cn}

\author[V.H Ya\~nez]{V\'{\i}ctor Hugo Ya\~nez}
\address{School of Mathematical Sciences and LPMC, Nankai University, Tianjin 300071, P.R. China}
\email{vhyanez@nankai.edu.cn}

\begin{abstract}
A topological group $G$ is said to have the \emph{L\'evy property} if it admits a dense subgroup which is decomposed as the union of an increasing sequence of compact subgroups $\mathcal{G} = \{G_i : i \in \N\}$ of $G$ which exhibits concentration of measure in the sense of Gromov and Milman. We say that $G$ has the \emph{strong L\'evy property} whenever the sequence $\mathcal{G}$ is comprised of finite subgroups. In this paper we give several new classes of isometry groups and countable topological groups with the strong L\'evy property. 
We prove that if $\Delta$ is a countable distance value set with  arbitrarily  small values, then $\Iso(\U_\Delta)$, the isometry group of the Urysohn $\Delta$-metric space equipped with the pointwise convergence topology, where $\U_\Delta$ is equipped with the metric topology, has the strong L\'evy property. We also prove that if $\Lan$ is a Lipschitz continuous signature, then $\Iso(\U_\Lan)$, the isometry group of the unique separable Urysohn $\Lan$-structure, has the strong L\'evy property. In addition, our approach shows that any countable omnigenous locally finite group can be given a topology with the L\'evy property. As a consequence to our results, we obtain at least continuum many pairwise nonisomorphic countable topological groups or isometry groups with the strong L\'evy property. 
\end{abstract}

%

%

%

%

%

%

%


\thanks{The authors acknowledge the partial support of their research by the Fundamental Research Funds for the Central Universities and by the National Natural Science Foundation of China (NSFC) grant 12271263. The third listed author also acknowledges support from the China Postdoctoral Science Foundation Tianjin Joint Support Program under
Grant Number 2025T001TJ}
\maketitle

\section{Introduction}

All topological spaces appearing in this paper are assumed to be Hausdorff. A topological group $G$ is \emph{extremely amenable} if every continuous action of $G$ on a compact space admits a fixed point. This is a remarkable property for a topological group to have, particularly in view of the theorem of Veech \cite{Veech} which implies that no locally compact groups can be extremely amenable. In fact, the question of whether extremely amenable groups even exist was only first posed explicitly by Mitchell \cite{Mitchell} in 1970, almost twenty years after the settlement of Hilbert's fifth problem. %
The first examples of extremely amenable groups were constructed by Herer and Christensen \cite{Herer} in the form of the so-called exotic groups. A topological group $G$ is \emph{exotic} if it admits no nontrivial unitary representations. The groups constructed in \cite{Herer} have the form $L_0(\mu, \mathbb{R})$, the additive group of all real-valued $\mu$-measurable functions with the topology of convergence in measure, where $\mu$ is a pathological submeasure. Since then, many groups of the form $L_0(\mu, G)$, where $\mu$ is a submeasure and $G$ is a topological group, have been found to be extremely amenable. This thread of research includes the following work:
\begin{enumerate}
\item[(a)] Glasner \cite{Glasner} and, independently, Furstenberg and Weiss (in  unpublished  work), for $\mu$ the Lebesgue measure on the unit interval $[0,1]$ and $G$ the circle group $\mathbb{T}$;
\item[(b)] Pestov \cite{Pestov2002} for $\mu$ the Lebesgue measure on the unit interval $[0,1]$ and $G$ any amenable locally compact second countable group;
\item[(c)] Farah and Solecki \cite{FarahSolecki} for $\mu$ any diffused submeasure and $G$ any solvable compact second countable group, and for $\mu$ any strongly diffused submeasure and $G$ any compact group or any amenable locally compact second countable group;
\item[(d)] Sabok \cite{Sabok} for $\mu$ any diffused submeasure and $G$ any solvable group.
\end{enumerate}

Historically, though, after Herer and  Christensen's  work \cite{Herer}, the next examples of extremely amenable groups were discovered by Gromov and Milman \cite{Gromov} through a concentration-of-measure phenomenon exhibited by certain groups, which is the main focus of this paper.  Following Glasner \cite{Glasner}, we say that a topological group $G$ satisfy  \emph{the L\'evy property} (or is a \emph{L\'evy group}) if
it admits an increasing sequence $\mathcal{G} = \{K_i : i \in \N\}$ of compact subgroups (equipped with the normalized Haar measures $\mu_i$, regarded as measures on $G$) with the following properties: 
\begin{itemize}
\item[(i)] The union $\bigcup \mathcal{G}$ is dense in $G$, and
\item[(ii)] For each open neighborhood $V$ of the identity of $G$ and each sequence $\{A_i : i \in \N\}$ of Borel subsets $A_i \subseteq K_i$ satisfying 
$$\liminf \mu_i(A_i) >0,$$we have
\begin{equation}
\lim \mu_i(VA_i) =1.
\end{equation}
\end{itemize}
We call the sequence $\mathcal{G}$ in the above definition a \emph{witnessing sequence} for the L\'evy property of $G$. Gromov and Milman \cite{Gromov} showed that the L\'evy property implies extreme amenability, and they gave several important examples of groups with the L\'evy property. Among them, they showed that the unitary group $U(H)$, where $H$ is  a separable infinite-dimensional Hilbert space, has the L\'evy property. By definition, $U(H)$ is not exotic; thus it is extremely amenable for a different reason other than exoticity. They also gave two examples of countable metric groups with the L\'evy property, which we will review in Section \ref{section:examples} below. In fact, in the above list of $L_0$-type groups known to be extremely amenable, many turn out to have the L\'evy property, which we also review in Section~\ref{section:examples}. Farah and Solecki \cite{FarahSolecki} also gave examples of extremely amenable $L_0$-type groups which do not have the L\'evy property (see our review in Section \ref{section:negating:levy} below).

Other than $L_0$-type groups, more operator groups were shown to be extremely amenable, some even to have the L\'evy property, after Gromov and Milman's work on $U(H)$. Most noteworthy among them is the work of Pestov  showing that $\mbox{Iso}(\mathbb{U})$, the group of all isometries of the universal Urysohn metric space $\mathbb{U}$ with the pointwise convergence topology, is extremely amenable \cite{Pestov2002}, and in fact has the L\'evy property \cite{Pestov1}. In this paper, we give more examples of isometry groups with the L\'evy property. Pestov \cite{Pestov1998} also showed the extremely amenability of $\mbox{Homeo}_+([0,1])$, the group of all orientation-preserving autohomeomorphisms of the unit interval with the pointwise convergence topology, by using a Ramsey type argument. Giordano and Pestov \cite{GiordanoPestov, GiordanoPestov2} showed, among other things, that $\mbox{Aut}(X, \mu)$, the group of all measure-preserving automorphisms of a standard nonatomic finite or sigma-finite measure space $(X, \mu)$, equipped with the weak topology, has the L\'evy property, and therefore is extremely amenable.

A third class of topological groups known to be extremely amenable in the literature comes from automorphism groups of countable structures. Pestov \cite{Pestov1998} gave the first example of such a group, which is $\mbox{Aut}(\mathbb{Q}, <)$, the group of all automorphisms of a countable dense linear order without endpoints, equipped with the pointwise convergence topology, where the underlying countable structure is equipped with the discrete topology. This was proved using a Ramsey type argument, which was then broadened in the work of Kechris, Pestov, and  Todor\v{c}evi\'{c}  \cite{KPT} to establish a complete characterization of all automorphism groups of countable structures which are extremely amenable. This work has been used to establish the extreme amenability of many automorphism groups of countable structures (see Nguyen Van Th\'e \cite{NVT} for a survey). However, no groups of this type were found to have the L\'evy property. In this paper, we isolate an argument of Schneider \cite{Schneider} to clarify that no non-Archimedean Polish group can have the L\'evy property, thus explaining the lack of examples of L\'evy groups in this widely studied class.

Among the wide variety of groups which have been shown to have the L\'evy property, a substantial number of them have a witnessing sequence consisting of finite groups. This motivates the following definition. We say that a topological group $G$ satisfy \emph{the strong L\'evy property} if there is an increasing sequence of finite groups, equipped with the normalized counting measures, witnessing the L\'evy property of $G$. Note the trivial fact that a group with the strong L\'evy property cannot be torsion-free. Also, for a countable topological group $G$, the L\'evy property and the strong L\'evy property are equivalent. One of the focuses of this paper is to study countable locally finite groups which can be given a group topology with the (strong) L\'evy property. Recall that a group $G$ is \emph{locally finite} if any finitely generated subgroup of $G$ is finite.

The main goal of this paper is to give more examples of groups with the L\'evy property. A secondary goal is to analyze the isomorphism types of these examples as topological groups, and to justify the claim that there are many of them which are pairwise nonisomorphic to each other. In doing this, we concentrate on isometry groups and countable groups. As mentioned above, Gromov and Milman \cite{Gromov} gave the first two examples of countable groups with the L\'evy property. Algebraically, their examples are, respectively, the direct sum of countably many copies of $\mathbb{Z}/2\mathbb{Z}$ and the group of all permutations of $\mathbb{N}$ which have finite support. Here, we give several new classes of examples of countable groups with the L\'evy property. In particular, we show that any countable  omnigenous  locally finite group carries a group topology with the L\'evy property. In view of the result of Gao and Li \cite{Gao-Li} which states that the isomorphism relation of countable  omnigenous  locally finite groups is Borel complete, this shows that there are many countable groups with the L\'evy property that are provably pairwise nonisomorphic to each other.

We also give more examples of isometry groups with the L\'evy property. The first class of these have the form $\mbox{Iso}(\mathbb{U}_\Delta)$, where $\Delta$ is a countable distance value set with $\inf \Delta=0$ and $\mathbb{U}_\Delta$ is the countable Urysohn $\Delta$-metric space, equipped with the topology of pointwise convergence in metric. These groups were known to be extremely amenable by the work of Etedadialiabadi, Gao and Li \cite{Eted-Gao-Feng}, and it is known that they are separable metrizable but not Polish.  The second class consists of Polish groups arising from continuous model theory. They are denoted $\mbox{Iso}(\mathbb{U}_{(\mathbb{Q}, V)})$, where $(\mathbb{Q}, V)$ is a countable good value pair for a Lipschitz continuous signature $\mathcal{L}$ with only finitely many relation symbols, but the group is in fact the automorphism group of the continuous Urysohn $\mathbb{Q}$-metric $\mathcal{L}$-structure $\mathbb{U}_{(\mathbb{Q}, V)}$. The structures $\mathbb{U}_{(\mathbb{Q}, V)}$ were shown to exist by Gao and Ren \cite{Gao-Ren}, and their automorphism groups are closed subgroups of $\mbox{Iso}(\mathbb{U})$, and therefore Polish.

The rest of the paper is organized as follows. 

In Section \ref{section:examples} we give a review of the previously known examples of topological groups with the L\'evy property. %
In Section \ref{section:negating:levy} we discuss some currently known methods for disproving the L\'evy property for a topological group. In particular, we discuss in detail the method of Schneider \cite{Schneider} which proves that a non-Archimedean Polish group cannot have the L\'evy property. In Section~\ref{section:newcountable} we define two new classes of countable metric groups with the L\'evy property. %

 In Section \ref{section:martingale} we accomplish two things. First, 
we develop a technical tool about adjusting distance values in a finite $\Delta$-metric space without changing its isometry group. 
Afterwards, we apply the  martingale  technique to deal with a technical problem that arises in the construction of witnessing sequences of the strong L\'evy property in the context of $\Delta$-metric spaces. 
With the technical tools in hand, we prove our main results in Sections \ref{first:density:section} through \ref{section:cont:logic}. In Section \ref{first:density:section} we use an elementary density argument to prove that a canonical isometric copy of the Hall's group in $\mbox{Iso}(\mathbb{U}_{\Delta})$ %
has the L\'evy property. In Section \ref{section:omnigen:construction} we give a recursive construction for proving that all countable  omnigenous  locally finite groups can witness the L\'evy property of $\Iso(\U_\Delta)$, and at the same time they themselves become L\'evy groups with the subspace topology.  %
Eventually, Section \ref{section:cont:logic} contains our results for the setting in continuous logic. 

In the final Section \ref{section:complexity} we study the isomorphism types of the L\'evy groups considered in this paper from the point of view of descriptive set theory of equivalence relations. 

\emph{Acknowledgments.} We thank Mahmood Etedadialiabadi, Aristotelis Panagiotopoulos, Frederich Martin Schneider and  S\l awek  Solecki for helpful discussions on the topic of the paper.

\section{A Review of Known Constructions} \label{section:examples}
In this section we review previously known constructions of topological groups with the L\'evy property.

\subsection{The Gromov--Milman constructions}

The following examples were constructed in the original 1983 paper \cite{Gromov} by Gromov and Milman.

\begin{example}[{\cite[Gromov--Milman]{Gromov}}]\label{ex:1} Let $E_2^n$ be the direct sum of $n$ copies of the group $\mathbb{Z}/2\mathbb{Z}$. Let $\theta_n\colon E_2^n\to E_2^{n+1}$ be the group embedding
$$ \theta_n(g_1, \dots, g_n)=(g_1,\dots, g_n, 0). $$
Let $E_2^\infty$ be the direct limit of the direct system $(E_2^n, \theta_n)$, i.e., each element of $E_2^n$ can be viewed as an  infinite  sequence $(g_n)$, where for all but finitely many $n$, $g_n=0$. For $g=(g_n), h=(h_n)\in E_2^\infty$, define
$$ \rho(g,h)=|\{n\colon g_n\neq h_n\}| $$
and let
$$ \varphi(g, h)=\displaystyle\frac{\rho(g, h)}{\max\{ \rho(g, \vec{0}), \rho(h, \vec{0})\}}, $$
where $\vec{0}=(0)$, if $g\neq h$, and let $\varphi(g, h)=0$ if $g=h$.
Finally, define a metric 
$$ \delta(g, h)=\inf\left\{ \displaystyle\sum_{i=0}^{m-1} \varphi(k_i, k_{i+1})\colon g=k_0, k_1, \dots, k_m=h\in E_2^\infty\right\}. $$
Then $(E_2^\infty, \delta)$ is a group with the L\'evy property.
\end{example}

\begin{example}[{\cite[Gromov--Milman]{Gromov}}]\label{ex:2} Let $S_n$ be the symmetric group of the $n$-element set $\{0, \dots, n-1\}$. Let $\theta_n\colon S_n\to S_{n+1}$ be the group embedding defined by
$$ \theta_n(g)(i)=\left\{\begin{array}{ll} g(i), & \mbox{if $i<n$,} \\ n, & \mbox{if $i=n$.} \end{array}\right. $$
Let $S_\omega$ be the direct limit of the direct system $(S_n, \theta_n)$. Define functions $\rho$, $\varphi$ and metric $\delta$ in the same fashion as in the previous example, except that we redefine $\vec{0}$ to be the trivial permutation. Then $(S_\omega, \delta)$ is a group with the L\'evy property.
\end{example}

\begin{example}[{\cite[Gromov--Milman]{Gromov}}; see also {\cite[Corollary 4.1.5]{Pestov2}}]\label{ex:3} The unitary group $U(H)$, where $H$ is a separable infinite-dimensional Hilbert space, equipped with the strong operator topology, has the L\'evy property.
\end{example}

In fact, Gromov and Milman essentially showed that any group that can be approximated by the sequence of unitary groups $U(n)$ of finite rank indeed has the L\'evy property. The following is an example from Pestov \cite{Pestov2}.

\begin{example}[{\cite[Corollary 4.1.17]{Pestov2}}]
The group of unitary operators of the form $\mathbb{I} + K$, where $\mathbb{I}$ is the identity and $K$ is a compact operator, equipped with the uniform operator topology, has the L\'evy property. 
\end{example}

\subsection{$L_0$-type groups}

Consider the notion of submeasure as in Herer and Christensen \cite{Herer}. Given a set $X$ and an algebra $\mathcal{B}$ of subsets of $X$,  a real-valued map $\phi: X \to \R$ is said to be a \emph{submeasure} if $\phi(\varnothing) = 0$ and for each pair $U,V \in \mathcal{B}$ one has 
\begin{itemize}
\item[(a)] $\phi(U) \leq \phi(V)$ whenever $U \subseteq V$, and
\item[(b)] $\phi(U \cup V) \leq \phi(U) + \phi(V)$.
\end{itemize}
A submeasure $\phi$ is called a \emph{measure} whenever item (b) above is an equality for disjoint sets $U,V \in \mathcal{B}$. Following Farah and Solecki \cite{FarahSolecki}, a submeasure $\phi$ is said to be \emph{diffused} if for each $\epsilon > 0$ there exists a covering of $X$ by sets from $\mathcal{B}$ of $\phi$-submeasure less than $\epsilon$, and $\phi$ is \emph{strongly diffused} if for each $\epsilon > 0$ there exists some $\delta > 0$ so that for any positive integer $N$ there exists a finite partition $\mathcal{P}\subseteq \mathcal{B}$ of $X$ with $|\mathcal{P}|>N$, such that for each $\mathcal{A} \subseteq \mathcal{P}$ satisfying $|\mathcal{A}| < \delta \cdot |\mathcal{P}|$ one has that $\phi(\bigcup \mathcal{A}) < \epsilon$. Note that each diffused measure is strongly diffused.  

Let $G$ be a topological group and $\phi$ be a submeasure (on a set $X$ with an algebra $\mathcal{B}$). We shall say that a function $f$ from $X$ to $G$ is a \emph{$G$-step function} %
whenever $f$ takes finitely many values in $G$, and the pre-images of points under $f$ are elements in $\mathcal{B}$. %
Given $G$-step functions $f,g$, we define $f\sim g$ 
if 
$$\phi(\{x\in X\colon f(x)\neq g(x)\})=0.$$ Then $\sim$ is an equivalence relation on the family of all $G$-step functions. We let $S(\phi,G)$ be the space of all equivalence classes of $G$-step functions under $\sim$, which we regard  %
as a topological group with the topology of
 \emph{convergence in measure}, 
for which a basis of the identity is given by the sets \begin{equation*}
O[V,\epsilon] := \left\{f \in S(\phi,G) \colon \phi(\left\{x \in X \colon f(x) \not \in V \right\} ) < \epsilon\right\},
\end{equation*}    
where $V$ is an arbitrary open neighborhood of the identity of $G$ and $\epsilon > 0$. Whenever $G$
admits a right-invariant metric $\rho$,  %
the group $S(\phi,G)$ admits a right-invariant metric in the form \begin{equation*}
d_\phi(f,g) = \inf \left\{ \epsilon > 0 \colon \phi(\left\{x \in X \colon \rho(f(x),g(x)) > \epsilon \right\} ) < \epsilon\right\}.
\end{equation*} 
We let $L_0(\phi,G)$ be the completion of $S(\phi,G)$ with respect to the metric $d_\phi$.

\begin{example}[{\cite[Glasner, Furstenberg--Weiss]{Glasner}}]
Let $\lambda$ be the Lebesgue measure on the unit interval $[0,1]$ and let $\mathbb{T}$ be the circle group. Then $L_0(\lambda, \mathbb{T})$ has the L\'evy property. 
\end{example}

\begin{example}[{\cite[Farah--Solecki]{FarahSolecki}}] \label{farahscolecki:l0:example}
Let $\phi$ be a strongly diffused submeasure. If $G$ is a compact group, then $L_0(\phi,G)$ satisfies the L\'evy property. 

\end{example}

\subsection{The Giordano--Pestov constructions}

Giordano and Pestov \cite{GiordanoPestov, GiordanoPestov2} showed that a number of operator groups have the L\'evy property by considering estimates of the concentration function, a technique we will use in our proofs in this paper.

\begin{example}[{\cite[Giordano--Pestov]{GiordanoPestov}}] 
Let $M$ be an injective von Neumann algebra, and let $U(M)_S$ be its unitary group endowed with the so-called $s(M, M_*)$-topology. Then $U(M)_S$ has the L\'evy property.
\end{example}

\begin{example}[{\cite[Giordano--Pestov]{GiordanoPestov2}}] 
Let $\mathrm{Aut}(X,\mu)$ be the group of all measure-preserving automorphisms of a standard nonatomic finite or sigma-finite measure space, equipped with the weak topology. Then $\mathrm{Aut}(X,\mu)$ has the strong L\'evy property. %
\end{example}

\begin{example}[{\cite[Giordano--Pestov]{GiordanoPestov2}}] 
Let $E$ be a discrete measured equivalence relation, acting ergodically on a Lebesgue space $(S, \mu)$. Suppose $E$ is an amenable equivalence relation. Then the full group $[E]$, the group of all bimeasurable transformations $\sigma$ of $(S, \mu)$ where $\sigma(s)Es$ for $\mu$-a.e. $s\in S$, equipped with the uniform topology, has the strong L\'evy property.
\end{example}

\subsection{The isometry group $\Iso(\mathbb{U})$}

Our results in this paper are motivated by the following influential result of Pestov \cite{Pestov1}.

\begin{example}[{\cite[Pestov]{Pestov1}}]
Let $\mathbb{U}$ be the universal Urysohn metric space and let $\Iso(\U)$ be the group of all auto-isometries of $\U$, equipped with the pointwise convergence topology. Then $\Iso(\U)$ has the strong L\'evy property. 
\end{example}

Related to this result, Pestov \cite{Pestov2002} had previously shown that $\Iso(\U)$ is extremely amenable and that it contains a dense locally finite subgroup. This latter fact was already nontrivial and had been independently proved by Solecki, Rosendal (see \cite{Rosendal}), and Vershik (see \cite{Vershik}).

\section{Known Methods for Proving Non-L\'evy-ness} \label{section:negating:levy}

Pestov has noted that $\mbox{Aut}(\mathbb{Q}, <)$ does not have the L\'evy property (see \cite[Page 7]{Pestov2}) since it does not contain any nontrivial compact subgroup. This result will be superseded by Corollary~\ref{cor:nonLevy} below, but for a different reason.

The first explicit method developed for disproving the L\'evy property was given by Farah and Solecki \cite{FarahSolecki}. The way the method works is by proving that any group $G$ satisfying the L\'evy property must also satisfy a ``minimum'' amount of measure concentration when one considers an arbitrary increasing sequence comprised of its compact subgroups. 

\begin{lemma}{\cite[Lemma 4.1]{FarahSolecki}}
Let $G$ be a topological group satisfying the L\'evy property. Let $\{K_n\colon n \in \N\}$ be an increasing sequence (by inclusion) of compact  subgroups  of $G$ with $\bigcup K_n$ dense in $G$. Let $V$ be an open neighborhood of the identity of $G$ and let $F \subseteq G$ be an arbitrary compact subset of $G$. Then, for any sequence $\{A_n\colon n \in \N\}$ where each $A_n$ is a Borel subset of $K_n$, we have the inequality \begin{equation*}
\liminf_{n} \mu_n(VA_n) - \mu_n(FA_n) \geq 0,
\end{equation*}
where $\mu_n$ is the probability Haar measure on $K_n$, considered as a measure on $G$ with support $K_n$.
\end{lemma}

They used the above lemma to show that certain $L_0$-type groups do not have the L\'evy property.

\begin{theorem}[Farah--Solecki {\cite[Theorem 4.2]{FarahSolecki}}] 
Let $G$ be any compact topological group. If $G$ is not connected, then there exists a sequence of natural numbers $\{M_i\colon i \in \N^+\}$ and a diffused submeasure $\psi_0$ on the algebra of all clopen subsets of the space $\prod_{i \in \N^+} M_i$ such that the group $L_0(\psi_0,G)$ \emph{fails} to have the L\'evy property.
\end{theorem}

In contrast, Sabok \cite[Theorem 1]{Sabok} has shown that the group $L_0(\phi, G)$, where $\phi$ is any diffused submeasure and $G$ is any abelian topological group, is always extremely amenable. Thus there is an abundance of examples of extremely amenable groups which do not have the L\'evy property. For instance, for any profinite abelian group $G$, there exists a diffused submeasure $\psi_0$ such that $L_0(\psi_0, G)$ is such a group.

Let us now discuss a method of Schneider \cite{Schneider} for proving that a group fails to have the L\'evy property. This method involves a concept of \emph{measure dissipation}. This concept was originally introduced for a sequence of mm-spaces (see \cite[Lemma 3.4]{Schneider}) and the more general case of uniform spaces (see \cite[Definition 3.12]{Schneider}), and here we employ a particular case of it for topological groups.

\begin{definition} \label{this:dissipate:def}
Let $G$ be a topological group. Let $\mathcal{G} = \{K_n\colon n \in \N^+\}$ be a sequence of compact subgroups of $G$, %
with $\mu_n$ being the normalized Haar measure on each $K_n$. For a open neighborhood $V$ of the identity of $G$, the sequence $\mathcal{G}$ is said to be \emph{$V$-dissipating} if there exists a collection of finite  families  $\{\mathcal{B}_n\colon n \in \N^+\}$  of Borel subsets of each corresponding $K_n$, %
with the following properties: \begin{itemize}
\item[(i)] for each $n \in \N^+$ and any two distinct $B,C \in \mathcal{B}_n$ we have $VB \cap C = \varnothing$ (we say that $B$ is \emph{$V$-separated} from $C$),
\item[(ii)] $\lim_n \mu_n ( \bigcup \mathcal{B}_n) = 1$; and
\item[(iii)] $\lim_n\sup\{\mu_n(B)\colon B \in \mathcal{B}_n\} = 0$. 
\end{itemize}  
$\mathcal{G}$ is said to be \emph{dissipating} if it is $V$-dissipating for some open neighborhood $V$ of the identity of $G$.
\end{definition}

The following is essentially proved in \cite[Lemma 3.4]{Schneider}.

\begin{lemma} \label{this:dissipation:translation}
Let $G$ be a topological group, let $V$ be an open neighborhood of the identity of $G$, and let $\mathcal{G} =  \{K_n\colon n \in \N^+\}$ be a $V$-dissipating sequence of compact subgroups of $G$, with $\mu_n$ being the normalized Haar measure on each $K_n$. Then, there exist $M \in \N^+$ and a sequence $\{A_n\colon n\geq M\}$ of Borel subsets of the corresponding $K_n$, such that
\begin{itemize}
\item[(a)] $\mu_n(A_n) \geq 1/2$ for all $n \geq M$, and
\item[(b)] $\limsup_{n} \mu_n(V A_n) \leq 3/4$.
\end{itemize}
\end{lemma}

\begin{proof} Fix three positive integers $p_1, p_2, q\in \N^+$ so that $p_1+p_2<q-2$, $2p_1>q$ and $4p_2>q$. Let $\{\mathcal{B}_n\colon n\in\N^+\}$ witness that $\mathcal{G}$ is $V$-dissipating.
Then there is $M \in \N^+$ such that for all $n \geq M$ \begin{equation} \label{eq:2}
\mu_n \left( \bigcup \mathcal{B}_n\right) \geq 1 - \frac{1}{q} \text{ and } \sup\left\{\mu_n(B)\colon B \in \mathcal{B}_n\right\} < \frac{1}{q}.
\end{equation}
\begin{claim} \label{claim:1}
For any $n\geq M$ and $i = 1,2$, there is a subfamily $\mathcal{C}_{n,i} \subseteq \mathcal{B}_n$ such that \begin{equation}\label{eq:3}
\frac{p_i}{q} \leq \mu_n\left(\bigcup \mathcal{C}_{n,i}\right) < \frac{p_i +1}{q}
\end{equation}
and for any pair $(B_1,B_2) \in \mathcal{C}_{n,1} \times \mathcal{C}_{n,2}$ the set $B_1$ is $V$-separated from $B_2$.
\end{claim}
\begin{proof}[Proof of Claim \ref{claim:1}] Fix $n\geq M$. Note that the first part of (\ref{eq:2}) gives $\mu(\bigcup\mathcal{B}_n)\geq (p_1+1)/q+(p_2+1)/q$. Let $\mathcal{C}_{n,1}$ be a subfamily of $\mathcal{B}_n$ of minimal cardinality such that \begin{equation}\label{eq:4}
\frac{p_1}{q} \leq \mu_n \left( \bigcup \mathcal{C}_{n,1}\right).
\end{equation}
Then the second part of (\ref{eq:2}) gives that $\mu_n(\bigcup \mathcal{C}_{n,1})<(p_1+1)/q$, since otherwise we get a subfamily of lower cardinality satisfying (\ref{eq:4}), a contradiction. Thus $\mathcal{C}_{n,1}$ satisfies (\ref{eq:3}).

Now let $\mathcal{D}_n = \mathcal{B}_n \setminus \mathcal{C}_{n,1}$. Since $\mu_n ( \bigcup \mathcal{C}_{n,1}) < (p_1+1)/q$, we get that $\mu(\mathcal{D}_n)>p_2/q$. %
Now let $\mathcal{C}_{n,2}$ be a subfamily of $\mathcal{D}_n$ of minimal cardinality such that 
\begin{equation*}
\frac{p_2}{q} \leq \mu_n \left( \bigcup \mathcal{C}_{n,2}\right).
\end{equation*}
Then a similar argument shows that (\ref{eq:3}) holds for $\mathcal{C}_{n,2}$. Since any two distinct members of $\mathcal{B}_n$ are $V$-separated from each other, and $\mathcal{C}_{n,1}$ and $\mathcal{C}_{n,2}$ are disjoint subfamilies of $\mathcal{B}_n$, the second part of Claim~\ref{claim:1} holds.
\end{proof}

For each $n\geq M$, fix subfamilies $\mathcal{C}_{n,1}$ and $\mathcal{C}_{n,2}$ satisfying Claim~\ref{claim:1}, and let
\begin{equation*}
A_n = \bigcup \mathcal{C}_{n,1} \text{ and } B_n = \bigcup \mathcal{C}_{n,2}.
\end{equation*}
We have that for each $n\geq M$, $A_n$ and $B_n$ are Borel and $V$-separated, i.e., $VA_n\cap B_n=\varnothing$. 
Moreover, $\mu_n(A_n)\geq p_1/q>1/2$ and $\mu_n(B_n)\geq p_2/q>1/4$. 
Since $VA_n\cap B_n=\varnothing$, we also have $\mu_n(VA_n)<1-\mu_n(B_n)<3/4$ for all $n\geq M$. Thus $\limsup_n\mu_n(VA_n)\leq 3/4$.
\end{proof}

The following is an immediate corollary of Lemma~\ref{this:dissipation:translation}.

\begin{proposition} \label{prop:3.5}Let $G$ be a topological group. Let $\mathcal{G}=\{K_n\colon n\in\N^+\}$ be a sequence of compact subgroups of $G$. Suppose $\bigcup K_n$ is dense in $G$ and $\mathcal{G}$ is dissipating. Then $G$ does not have the L\'evy property.
\end{proposition}

The following is one of the main results of Schneider \cite{Schneider}, which states that the above hypotheses hold for metrizable groups which admit open subgroups of infinite index.

\begin{theorem}[Schneider {\cite[Corollary 5.3]{Schneider}}] \label{schneider:thm}

If $G$ is a metrizable group with an open subgroup of infinite index, then any increasing sequence $\{K_n: i \in \N^+\}$ of compact subgroups of $G$ with $\bigcup K_n$ dense in $G$ is dissipating.
\end{theorem}

Recall that a topological group $G$ is \emph{Polish} if it is separable and completely metrizable, $G$ is \emph{non-Archimedean} if it admits a neighborhood basis of the identity which consists of open subgroups of $G$, and $G$ is \emph{precompact} (or \emph{totally bounded}) if for each open neighborhood $U$ of the identity of $G$, there exists a finite set $F \subseteq G$ such that $UF=FU=G$ (we say that $U$ is of \emph{finite index} in $G$).

A  prototypical  example of a non-Archimedean Polish group is $S_\infty$, the group of all permutations of $\mathbb{N}$ with the pointwise convergence topology (where $\mathbb{N}$ is equipped with the discrete topology). By a classical result of Becker and Kechris (see \cite{BeckerKechris}; also see \cite[Section 2.4]{GaoBook}), a topological group $G$ is a non-Archimedean Polish group if and only if $G$ is isomorphic to a closed subgroup of $S_\infty$, and if and only if $G$ is isomorphic to an automorphism group of a countable structure. 

Also note that a topological group $G$ is precompact if and only if it is isomorphic to a subgroup of a compact group (see \cite[Corollary 3.7.17]{TkachenkoBook}). If $G$ is a non-Archimedean group, then $G$ is not precompact if and only if it admits an open subgroup of infinite index.

Combining all of the above, we obtain the following corollary.

\begin{corollary}\label{cor:nonLevy} No non-Archimedean Polish groups can have the L\'evy property.
\end{corollary}

\begin{proof} Toward a contradiction, assume that $G$ is a non-Archimedean Polish group with the L\'evy property. Then $G$ is extremely amenable. In view of Theorem~\ref{schneider:thm} and Proposition~\ref{prop:3.5}, it suffices to show that $G$ is not precompact. However, if $G$ were precompact, then as an isomorph of a closed subgroup of $S_\infty$, it is itself a complete group and therefore compact (since its completion would be compact). Consider the action of $G$ on itself by left translation. This action does not have a fixed point. Hence $G$ cannot be extremely amenable, a contradiction.
\end{proof}

 Equivalently,  no closed subgroups of $S_\infty$ or automorphism groups of countable structures (with the pointwise convergence topology, where the underlying countable domain is equipped with the discrete topology) can have the L\'evy property.

\section{Two Classes of Countable L\'evy Groups}\label{section:newcountable}

The following example is motivated by the Gromov--Milman Example~\ref{ex:1}.

\begin{example}\label{ex:3} Consider the directed set $D=(\mathbb{N}^+, \sqsubseteq)$, where $n\sqsubseteq m$ iff $n\,|\, m$. For each $n\in D$, let $E_2^n$ be the group defined in Example~\ref{ex:1}. If $n\sqsubseteq m$, let $\varphi_{n, m}\colon E_2^n \to E_2^m$ be the group embedding
$$ \varphi_{n,m}(g_1,\dots, g_n)=(g_1, \dots, g_n, g_1, \dots, g_n, \cdots, g_1, \dots, g_n). $$
Thus $\varphi_{n,m}$ is a diagonal embedding of $E_2^n$ into $E_2^m$. Let $Z_2^\infty$ be the direct limit of the direct system $\{E_2^n, \varphi_{n,m}\colon n, m\in D, n\sqsubseteq m\}$. Like $E_2^\infty$, $Z_2^\infty$ is a countable abelian group. For each $n\in D$ and $g=(g_i), h=(h_i)\in E_2^n$, define
$$ \rho_n(g, h)=|\{i\colon g_i\neq h_i\}| $$
as in Example~\ref{ex:1}. Then $\rho_n/n$ is a metric on $E_2^n$, and $\varphi_{n,m}$ is an isometric embedding from $(E_2^n, \rho_n/n)$ into $(E_2^m, \rho_m/m)$ when $n\sqsubseteq m\in D$. Thus we can define the following metric $\delta_\infty$ on $Z_2^\infty$. For $x,y\in Z_2^\infty$, let $n\in D$ be large enough and $g, h\in E_2^n$ be such that $x=\varphi_{n, \infty}(g)$ and $y=\varphi_{n, \infty}(h)$. Then define 
$$ \delta_\infty(x,y)=\displaystyle\frac{\rho_n(g,h)}{n}. $$
Now $(Z_2^\infty, \delta_\infty)$ is a group with the L\'evy property. 
\end{example}

To justify the L\'evy property of $Z_2^\infty$, we use the following concept of \emph{concentration function} and result (see Pestov \cite[Theorem 4.3.19]{Pestov2}) proved using a martingale technique. Similar technique and results can be traced back to the work of Milman--Schechtman \cite{Schechtman2} and Schechtman \cite{Schechtman1}. It is well known (see \cite[Exercise 1.3.1]{Pestov2}) that in considering the measure concentration, it is enough to focus on sets whose measures are bounded away from zero by an apriori chosen constant between 0 and 1. In the rest of this paper, we follow Pestov \cite{Pestov2} and use $1/2$.

\begin{definition}[{\cite[Definition 1.3.2]{Pestov2}}]\label{def:confun}
Given a metric-measure space  $(X,d,\mu)$, the \emph{concentration function of $X$} is a real-valued function $\alpha_X\colon  [0,\infty) \to [0,\infty)$ such that $\alpha_X(0) = 1/2$, and for each $\varepsilon > 0$, \begin{equation}
\alpha_X(\varepsilon) = 1 - \inf{\left\{\mu(B_\varepsilon): B \subseteq X, \mu(B) \geq 1/2\right\}},
\end{equation}
where $B_\varepsilon=\{y\in X\colon d(x,y)<\varepsilon \mbox{ for some $x\in B$}\}$.
\end{definition}

\begin{theorem}[{\cite[Theorem 4.3.19]{Pestov2}}]\label{thm:mar}
Let $(X_i,d_i,\mu_i)$, $i=1,2, \dots, n$, be metric-measure spaces, with $\diam X_i=a_i$. Equip the product space $X=\prod_{i=1}^n X_i$ with the product measure $\otimes_{i=1}^n \mu_i$ and the Hamming metric
$$ d(x,y)=\displaystyle\sum_{i=1}^n d_i(x_i, y_i).$$
Then the concentration function of $X$ satisfies
$$ \alpha_X(\varepsilon)\leq 2e^{-\varepsilon^2/(8\sum_{i=1}^n a_i^2)}. $$
\end{theorem}

By Theorem~\ref{thm:mar}, the concentration function of $X=(E_2^n, \rho_n/n, \mu_n)$ (where $\mu_n$ is the counting measure) satisfies $\alpha_X(\varepsilon)\leq 2e^{-n\varepsilon^2/8}$. Thus it goes to $0$ as $n\to\infty$. This shows that $(Z_2^\infty, \delta_\infty)$ has the L\'evy property.

We remark that the construction in Example~\ref{ex:3} yields the same group as given by the following linear direct system. Let $(n_k)$ be an increasing sequence of positive integers so that $n_k\,|\, n_{k+1}$ (such sequences are called {\em scales} in the literature of topological dynamics). Let $\mathfrak{n}$ be the supernatural number that is the least common multiple of $(n_k)$. The scale $(n_k)$ is said to be \emph{universal} if for any prime $p$, $p^\infty$ is a factor of $\mathfrak{n}$. Equivalently, $(n_k)$ is universal if and only if for any $m\in \mathbb{N}^+$ there is some $k$ such that $m\,|\, n_k$, and if and only if $(n_k)$ is a cofinal sequence in the directed set $(D, \sqsubseteq)$. If $(n_k)$ is a universal scale, the direct system $(E_2^{n_k}, \varphi_{n_k, n_{k+1}})$ gives rise to a direct limit which is isomorphic to $Z_2^\infty$. 

The next example comes from a similar but more general construction.

\begin{example}\label{ex:4} Let $\mathfrak{n}$ be any supernatural number. Consider the directed set $D=(N, \sqsubseteq)$, where $N$ is the set of all positive finite factors of $\mathfrak{n}$, and $\sqsubseteq$ is the factor relation as in Example \ref{ex:3}. As in Example~\ref{ex:3}, we define $\varphi_{n,m}$ for $n\sqsubseteq m$, $\rho_n$ for each $n\in D$, and a similar $\delta_\infty$.  Let $Z_2^{\mathfrak{n}}$ be the direct limit of the direct system $\{E_2^n, \varphi_{n,m}\colon n, m\in D, n\subseteq m\}$. Then $(Z_2^{\mathfrak{n}}, \delta_\infty)$ is a group with the L\'evy property.
\end{example}

If $(n_k)$ is any scale with $\mathfrak{n}$ as the least common multiple of $(n_k)$. Then the direct system $(E_2^{n_k}, \varphi_{n_k, n_{k+1}})$ gives a direct limit that is isomorphic to $Z_2^{\mathfrak{n}}$. The L\'evy property of $Z_2^{\mathfrak{n}}$ follows from a similar argument as for Example~\ref{ex:3}. Note that all examples in Examples~\ref{ex:1}, \ref{ex:3} and \ref{ex:4} are abelian and of exponent $2$, and they are all isomorphic to each other as algebraic groups. We do not know, however, if they are isomorphic as topological groups.

The following examples are motivated by Example~\ref{ex:2} above, but the constructions are similar to $Z_2^\infty$ and $Z_2^{\mathfrak{n}}$.

\begin{example}\label{ex:5} Let $(n_k)$ be any scale and let $\mathfrak{n}=\lcm(n_k)$. For each $k$, let $q_k=n_{k+1}/n_k$. Consider the direct system of finite groups
$$ G_0\stackrel{\varphi_0}{\longrightarrow} G_1\stackrel{\varphi_1}{\longrightarrow} G_2 \stackrel{\varphi_2}{\longrightarrow} \cdots \stackrel{\varphi_{k-1}}{\longrightarrow} G_k \stackrel{\varphi_k}{\longrightarrow} G_{k+1}\stackrel{\varphi_{k+1}}{\longrightarrow} \cdots $$
where $G_k=\mbox{Sym}(X_{n_k})$ is the symmetric group of the $n_k$-element set $X_{n_k}=\{0,\dots, n_k-1\}$, and $\varphi_k\colon G_k\to G_{k+1}$ is the group embedding defined by 
$$\varphi_k(g)(jn_k+i)=jn_k+g(i) $$
for all $j=0, \dots, q_k-1$ and $i=0, \dots, n_k-1$. Let $G_{\mathfrak{n}}$ be the direct limit. Define a metric $\rho_k$ on each $G_k$ as in Example~\ref{ex:3}. Then $\varphi_k$ is an isometric embedding from $(G_k, \rho_k/n_k)$ into $(G_{k+1}, \rho_{k+1}/n_{k+1})$. This allows us to define a metric $\delta_{\mathfrak{n}}$ on $G_{\mathfrak{n}}$. We have that $(G_{\mathfrak{n}}, \delta_{\mathfrak{n}})$ has the L\'evy property. 
\end{example}

We call a countable group $G$ \emph{universal for finite groups} if it contains an isomorphic copy of every finite group as a subgroup. The group $S_\omega$ in Example~\ref{ex:2} and the groups $G_{\mathfrak{n}}$ in Example~\ref{ex:5} are all universal for finite groups. In particular, they are nonabelian, and therefore are not isomorphic to the groups $E_2^\infty$, $Z_2^\infty$ and $Z_2^{\mathfrak{n}}$. 

We show below that $S_\omega$ is not isomorphic to any of the groups $G_{\mathfrak{n}}$. This is done by considering 
an algebraic property known as MIF. 

\begin{definition}[Hull--Osin {\cite[Definition 5.2]{HO}}]
Let $G$ be a group. Given a free group $F_n$ of rank $n$ generated by letters $v_1,\dots,v_n$, a \emph{non-trivial mixed identity} in the group $G$ is a word $w(v_1,\dots,v_n) \in (G * F_n) \setminus G$ such that $w(g_1,\dots, g_n) = e$ for all $g_1,\dots,g_n \in G$, where $e$ is the identity of $G$. If there are no non-trivial mixed identities in $G$, then we say that $G$ is \emph{mixed identity free} (\emph{MIF}).
\end{definition} 

Hull and Osin \cite[Theorem 5.9]{HO} showed that for a highly transitive countable subgroup $G$ of $S_\infty$, $G$ is not MIF if and only if $G$ contains a normal subgroup isomorphic to $A_\omega$, where $A_\omega$ is the subgroup of $S_\omega$ consisting of all even permutations. In particular, this implies that $S_\omega$ is not MIF. In contrast, we show below that $G_{\mathfrak{n}}$ is MIF for any infinite supernatural number $\mathfrak{n}$. 

\begin{proposition} For any scale $(n_k)$ with $\mathfrak{n}=\lcm(n_k)$, $G_{\mathfrak{n}}$ is MIF.
\end{proposition}

\begin{proof} Toward a contradiction, assume $G_{\mathfrak{n}}$ is  not  MIF, i.e. there exist 
$$x_1, \dots, x_m\in G_{\mathfrak{n}}$$ and a word $$w=w(v_1, \dots, v_\ell, x_1, \dots, x_m)\in \mathbb{F}(v_1, \dots, v_{\ell})*\langle x_1,\dots, x_\ell\rangle,$$ where $\mathbb{F}(v_1,\dots, v_\ell)$ is the free group generated by letters $v_1,\dots, v_\ell$, such that 
$$ w(y_1, \dots, y_\ell, x_1, \dots, x_m)=e $$
for all $y_1, \dots, y_\ell\in G_{\mathfrak{n}}$. In particular, for the word 
$$w'=w'(v, x_1,\dots, x_m)=w(v, \dots, v, x_1,\dots, x_m)\in \mathbb{F}(v)*\langle x_1, \dots, x_\ell\rangle, $$
where $v$ is a new letter, we have
$$ w'(y, x_1,\dots, x_m)=e $$
for all $y\in G_{\mathfrak{n}}$. Let
$$ w'=v^{\alpha_{p+1}}z_pv^{\alpha_p}\cdots  v^{\alpha_2}z_1v^{\alpha_1} $$
where for each $i=1,\dots, p, p+1$, $\alpha_i\in \mathbb{N}$, and for each $i=1,\dots, p$, $z_i\in \langle x_1,\dots, x_m\rangle$, and all of $z_1,\dots, z_p, v^{\alpha_1},\dots, v^{\alpha_p}, v^{\alpha_{p+1}}$ are non-identity elements except $v^{\alpha_1}$ or $v^{\alpha_{p+1}}$. Note that each of $z_1,\dots, z_p$ has infinite support. By \cite[Lemma 5.6]{HO}, there is a permutation $t\in \mbox{Sym}(\mathbb{N})$ and $n\in \mathbb{N}$ such that $w'(t, x_1,\dots, x_m)(n)\neq n$. Since the computation of $w'(t, x_1, \dots, x_m)(n)$ only uses finitely many natural numbers, we conclude that there is $y\in G_{\mathfrak{n}}$ such that $w'(y, x_1,\dots, x_m)(n)\neq n$, and thus $w'(y, x_1, \dots, x_m)\neq e$, a contradiction.
\end{proof}

Etedadialiabadi showed that $G_{\mathfrak{n}}$ is MIF with a slightly different proof.

In Section~\ref{first:density:section} we will show that the Hall's universal locally finite group $\mathbb{H}$, defined by Hall \cite{Hall}, can be given a group topology with the L\'evy property. Below we show that $\mathbb{H}$ is not isomorphic to $G_{\mathfrak{n}}$ for any infinite supernatural number $\mathfrak{n}$. 

We recall the definition and properties of $\mathbb{H}$; for more details, see \cite{Hall} and, e.g., \cite{Gao-Adv} or \cite{Eted-Gao-Feng}. Recall that $\mathbb{H}$ is the countable locally finite group $H$, unique up to isomorphism, which satisfies any of the following equivalent conditions:
\begin{itemize}
\item[(i)] $H$ is universal for finite groups, and $H$ is ultrahomogeneous (i.e., for any isomorphic finite subgroups $F_1, F_2$ of $H$, there is an automorphism $\varphi$ of $H$, such that $\varphi(F_1)=F_2$);
\item[(ii)] $H$ is universal for finite groups, and for any finite subgroups $F_1, F_2$ of $H$ and isomorphism $\theta\colon F_1\to F_2$, there is  $g\in H$ such that for all $x\in F_1$, $\theta(x)=g^{-1}xg$;
\item[(iii)] for any finite groups $F_1, F_2$ and group embeddings $\varphi\colon F_1\to H$ and $i\colon F_1\to F_2$, there is a group embedding $\Phi\colon F_2\to H$ such that $\varphi=\Phi\circ i$.
\end{itemize}
Hall \cite{Hall} showed the existence and the simplicity of $\mathbb{H}$. It is well known that $\mathbb{H}$ is isomorphic to the Fra\"iss\'e limit of the Fra\"iss\'e class of all finite groups (see, e.g.,  \cite{Gao-Adv}; also see \cite{Gao-Adv} for a brief review of the Fra\"iss\'e theory). It was shown by Etedadialiabadi, Gao, Le Ma\^itre, and Melleray \cite[Corollary 6.14]{Gao-Adv} that $\mathbb{H}$ is  MIF. 

\begin{proposition} For any scale $(n_k)$ with $\mathfrak{n}=\lcm(n_k)$, $G_{\mathfrak{n}}$ is not isomorphic to $\mathbb{H}$. 
\end{proposition}

\begin{proof}
To see that $G_{\mathfrak{n}}$ is not isomorphic to $\mathbb{H}$, we show that there are isomorphic finite subgroups $F_1$ and $F_2$ of $G_{\mathfrak{n}}$ so that $F_1$ and $F_2$ are not conjugate within $G_{\mathfrak{n}}$, and therefore $G_{\mathfrak{n}}$ does not satisfy (ii) above. For this, consider an arbitrary $k$ where $n_k>3$. Let $g\in G_k=\mbox{Sym}(X_{n_k})$ be the cyclic permutation defined by $g(a)=a+1\!\!\mod\! n_k$ for $a\in X_{n_k}$. Let $x=\varphi_{k,\infty}(g)$ and $F_1=\langle x\rangle\leq G_{\mathfrak{n}}$. Then $|F_1|=n_k$ and $F_1$ is cyclic. Moreover, it is easy to see that, considering $F_1$ as a permutation group on $\mathbb{N}$,  no element of $F_1$ has fixed points.  Now let $h\in G_{k+1}=\mbox{Sym}(X_{n_{k+1}})$  be  the permutation defined by
$$ h(a)=\left\{\begin{array}{ll} a+1\!\!\!\mod\! n_k, & \mbox{ if $a\in X_{n_k}=\{0,\dots, n_k-1\}$,} \\ a, & \mbox{ if $a\in X_{n_{k+1}}\setminus X_{n_k}$.}\end{array}\right. $$
Let $y=\varphi_{k+1,\infty}(h)$ and $F_2=\langle y\rangle\in G_{\mathfrak{n}}$. Then it is easy to verify that $F_2$ is also cyclic and $|F_2|=n_k$. Thus $F_1$ and $F_2$ are isomorphic subgroups of $G_{\mathfrak{n}}$. Moreover, $y$ has infinitely many fixed points by the construction of $h$. Since $G_{\mathfrak{n}}$ is a permutation group of $\mathbb{N}$, if $F_1$ and $F_2$ were conjugate in $G_{\mathfrak{n}}$, i.e., there was $r\in G_{\mathfrak{n}}$ such that $r^{-1}xr=y$, then $y$ would have no fixed points as a permutation of $\mathbb{N}$, a contradiction.
\end{proof}

Before closing this section we note that not all $G_{\mathfrak{n}}$s are algebraically isomorphic to each other. In fact, if $\mathfrak{n}$ has $2^\infty$ as a factor, then one can show that $G_{\mathfrak{n}}$ is simple, whereas if $\mathfrak{n}$ does not have $2^\infty$ as a factor, then $G_{\mathfrak{n}}$ contains a simple subgroup of index $2$. Thus, for instance, $G_{2^\infty}$ and $G_{3^\infty}$ are not isomorphic algebraically.

\section{$\Delta$-Metric Approximations on Finite Powers} \label{section:martingale}

Starting from this section, we develop the main results of this paper. They are all about the notion of $\Delta$-metric space for a countable distance value set $\Delta$. We first recall the definitions.

\begin{definition}[{\cite[Definition 2.5]{Gao-Adv}}]\label{def:dvs} A \emph{distance value set} is a nonempty set $\Delta$ of positive real numbers satisfying that for all $x, y\in \Delta$, $\min\{x+y, \sup\Delta\}\in \Delta$. Given a distance value set $\Delta$, a \emph{$\Delta$-metric space} is a metric space whose nonzero distances belong to $\Delta$.
\end{definition}

We also recall the following basic facts about $\Delta$-metric spaces from \cite{Gao-Adv}. If $\Delta$ is a countable distance value set, then the class of all finite $\Delta$-metric spaces is a Fra\"iss\'e class. We denote its Fra\"iss\'e limit by $\mathbb{U}_\Delta$, and call it the \emph{Urysohn $\Delta$-metric space}. Then $\U_\Delta$ is a countable $\Delta$-metric space with the following properties:
\begin{itemize}
\item $\U_\Delta$ is \emph{universal for finite $\Delta$-metric spaces}, i.e., it contains an isometric copy of every finite $\Delta$-metric space;
\item $\U_\Delta$ is \emph{ultrahomogeneous}, i.e., for any isometric finite subspace $F_1, F_2$ of $\U_\Delta$, there is an isometry $\varphi$ of $\U_\Delta$ such that $\varphi(F_1)=F_2$.
\end{itemize}
In fact, $\U_\Delta$ is, up to isometry, the unique countable $\Delta$-metric space with both these properties. 

The following facts are from Etedadialiabadi, Gao and Li \cite{Eted-Gao-Feng}. Let $\mbox{Iso}(\U_\Delta)$ denote the group of all isometries of $\U_\Delta$ equipped with the pointwise convergence topology, i.e. a basic open neighborhood of the identity of $G$ is of the form 
\begin{equation}\label{eq:nbhd}
V[x_1, \dots, x_n; \epsilon]=\left\{g\in \mbox{Iso}(\U_\Delta)\colon d_{\U_\Delta}(x_i,g(x_i))<\epsilon \mbox{ for $i=1,\dots, n$}\right\} 
\end{equation}
where $f\in \mbox{Iso}(\U_\Delta)$, $\epsilon>0$, $x_1, \dots, x_n\in \U_\Delta$, and $d_{\U_\Delta}$ is the metric on $\U_\Delta$. Then $\mbox{Iso}(\U_\Delta)$ is a separable metrizable group. It is Polish if and only if $\inf\Delta>0$, and it is extremely amenable if and only if $\inf \Delta=0$.

One of our main results in this paper is that $\mbox{Iso}(\U_\Delta)$ has the strong L\'evy property when $\inf\Delta=0$. We give two proofs for this fact. However, in both proofs, we have to deal with the following issue if we try to  mimic  Pestov's proof of the strong L\'evy property for $\mbox{Iso}(\U)$. In a general step of the construction of the witnessing sequence, Pestov considers the diagonal extension of a finite metric space $(X, d)$ to a finite power $X^m$ with the \emph{normalized Hamming metric}
$$ d'(x, y)=\displaystyle\frac{1}{m}\sum_{i=1}^m d(x_i, y_i). $$
In the $\Delta$-metric context this is no longer legitimate (since the values do not necessarily belong to $\Delta$ anymore) and we have to find $\Delta$-metric spaces that approximate the extension. The rest of this section is devoted to developing some tools to deal with this issue.

\subsection{$\Delta$-remetrization of finite metric spaces} \label{subsection:delta:triangles}

In this section we develop a tool to accomplish the following. Given a finite metric space $(X, d)$ and a countable distance value set $\Delta$ with $\inf\Delta=0$, we would like to reassign the distance values of $(X,d)$ by values from $\Delta$ so that the resulting $\Delta$-metric space $(X,d')$ has the same isometries as $\mbox{Iso}(X,d)$. Moreover, we  will maintain the following properties:
\begin{itemize}
\item if $d(x, y)\in \Delta$, then $d'(x,y)=d(x,y)$; and
\item for all $x, y\in X$, $|d'(x,y)-d(x,y)|<\epsilon$ for some predetermined $\epsilon>0$.
\end{itemize}
In the following, we are able to do this in a stronger way so that
\begin{itemize}
\item if $d(x,y)=d(x',y')$, then $d'(x,y)=d'(x', y')$.
\end{itemize}
That is, we are going to make the $\Delta$-remetrization by only working with the distance values of $(X, d)$ rather than with the $(X, d)$ itself.

\begin{definition}\label{def:remet} Let $\epsilon>0$ and let $\Delta$ be a distance value set. For a metric space $(X, d)$, a \emph{$(\Delta, \epsilon)$-remetrization} of $(X, d)$ is a $\Delta$-metric $d'$ on $X$ such that
\begin{enumerate}
\item[(a)] for all $x, y\in X$, if $d(x,y)\in \Delta$, then $d'(x,y)=d(x,y)$;
\item[(b)] for all $x, y\in X$, $|d'(x,y)-d(x,y)|<\epsilon$;
\item[(c)] for all $x,x',y,y'\in X$, if $d(x,y)=d(x',y')$, then $d'(x,y)=d'(x',y')$;
\item[(d)] $\Iso(X, d)=\Iso(X, d')$.
\end{enumerate}
\end{definition}

Our main theorem of this subsection is the following.

\begin{theorem} \label{this:delta:thm}
Let $\epsilon>0$ and let $\Delta$ be a distance value set with $\inf\Delta=0$. Then for any finite metric space $(X, d)$ with $\diam(X, d)\leq \sup\Delta$, there is a $(\Delta, \epsilon)$-remetrization $d'$ on $X$.
\end{theorem}

The rest of this subsection is devoted to a proof of Theorem~\ref{this:delta:thm}. We use the following notion of $S$-triangle in the proof.

\begin{definition}[{\cite[Definition 7.5]{Gao-Adv}}] \label{delta:triang:def}
Let $S$ be a subset of $\R^+$. We say that a triple $(s_1,s_2,s_3) \in S^3$ is an \emph{$S$-triangle} whenever the inequalities $|s_1-s_2| \leq s_3 \leq s_1+s_2$ hold.

\end{definition}

\begin{remark} \label{this:permutation:remark}
It is easily seen that a triple $(s_1,s_2,s_3) \in S^3$ is an $S$-triangle if and only if $(s_{\sigma(1)},s_{\sigma(2)},s_{\sigma(3)})$ is an $S$-triangle for any permutation $\sigma$ on $\{1,2,3\}$. 
\end{remark}

The following is our main technical lemma. 

\begin{lemma} \label{this:triangular:lemma}
Let $S$ and $S'$ be two finite subsets of $\mathbb{R}^+$ with the property:
\begin{quote}
if $s, t \in S \cap S'$ and $s + t \in S$ hold, then $s + t \in S \cap S'$. 
\end{quote}
Define the set \begin{equation} \label{set:gamma:fors}
\Gamma_S = \left\{r >0\colon \mbox{ for some $S$-triangle $(s_1,s_2,s_3)$, $r= s_1 + s_2 - s_3$} \right\},
\end{equation}
and let $0 < \gamma = \min \Gamma_S$. Let $t_i$, $1\leq i \leq m$, be an increasing enumeration of the elements of $S \setminus S'$. Assume there exist $0<\epsilon <\gamma/2$ and a bijection $\Phi: S \to S'$ with the following properties:

  \begin{itemize}
\item[(i)] $\Phi$ fixes the intersection $S \cap S'$ pointwise; 
\item[(ii)] for any $1\leq i \leq m$, we have 
$$\displaystyle\frac{\epsilon}{2^{i+1}} < \Phi(t_i) - t_i < \frac{\epsilon}{2^i}.$$  
\end{itemize}
Then, $(\Phi(s_1),\Phi(s_2),\Phi(s_3))$ is an $S'$-triangle whenever $(s_1,s_2,s_3)$ is an $S$-triangle. 
\end{lemma}

\begin{proof} First note that our assumptions imply that for any $s\in S$, $\Phi(s)\geq s>0$. 
Now assume that $(\Phi(s_1),\Phi(s_2),\Phi(s_3))$ fails to be an $S'$-triangle for some triple $(s_1,s_2,s_3)\in S^3$. This implies that $\Phi(s_3) > \Phi(s_1) + \Phi(s_2)$ or $\Phi(s_3) < |\Phi(s_1) - \Phi(s_2)|$. %
\begin{claim} \label{this:claim:1.3:1}
If $\Phi(s_3) > \Phi(s_1) + \Phi(s_2)$ holds, then $(s_1,s_2,s_3)$ is not an $S$-triangle.
\end{claim}
\begin{proof}[Proof of Claim~\ref{this:claim:1.3:1}]
We note that the assertion is true whenever the inequality $s_3 >  \Phi(s_1) + \Phi(s_2)$ holds. Indeed, in this case we have $s_3 >\Phi(s_1)+\Phi(s_2)\geq s_1+s_2$, and thus $(s_1,s_2,s_3)$ is not an $S$-triangle. We may therefore assume that \begin{equation}  \label{eq:7}
\Phi(s_3) > \Phi(s_1) + \Phi(s_2) \geq s_3.
\end{equation} 
By (i), this implies that $s_3 \in S \setminus S'$. It also follows from (\ref{eq:7}) that $s_3\neq s_1$ and $s_3\neq s_2$.
Toward a contradiction, assume $(s_1, s_2, s_3)$ is an $S$-triangle, and in particular $s_3\leq s_1+s_2$. We consider two cases. 

Case 1: $s_3 = s_1 + s_2$. We first argue that for at least one of $i \in\{1,2\}$, $s_i \in S \setminus S'$. Otherwise,  $s_1, s_2 \in S \cap S'$ and $s_1 + s_2 = s_3 \in S$ both hold, and by our assumption of $S$ and $S'$, we must have $s_3 \in S \cap S'$ and so $\Phi(s_3) = s_3$ by (i), contradicting \eqref{eq:7}. Without loss of generality, we may assume $d_1 \in S \setminus S'$.

We now have $s_3,s_1 \in S \setminus S'$ with $s_3 >s_1$. Thus there are $1\leq i< j\leq m$ such that $s_1 = t_i$ and $s_3 = t_j$.  By (ii), we have \begin{equation*}
\Phi(s_3) - s_3 < \frac{\epsilon}{2^j} \leq \frac{\epsilon}{2^{i+1}} < \Phi(s_1) - s_1.
\end{equation*}
Now we have
$$\begin{array}{rcl} s_1+s_2-s_3&\leq& s_1+\Phi(s_2)-s_3 \\
&=&s_1-\Phi(s_1)+\Phi(s_1)+\Phi(s_2)-s_3 \\
&<&s_1-\Phi(s_1)+\Phi(s_3)-s_3<0, 
\end{array}$$
contradicting our case assumption.

Case 2: $s_3 < s_1 + s_2$. Let $1\leq i\leq m$ be such that $s_3=t_i$. Then we have
$$ \gamma\leq s_1+s_2-s_3\leq \Phi(s_1)+\Phi(s_2)-s_3<\Phi(s_3)-s_3<\displaystyle\frac{\epsilon}{2^i}<\frac{\gamma}{2^{i+1}}<\gamma, $$
again a contradiction. 

These contradictions imply that $(s_1,s_2,s_3)$ is not an $S$-triangle.
\end{proof}

\begin{claim} \label{this:claim:1.3:2}
If $\Phi(s_3) < |\Phi(s_1) - \Phi(s_2)|$ then $(s_1,s_2,s_3)$ is not an $S$-triangle.
\end{claim}
\begin{proof}[Proof of Claim~\ref{this:claim:1.3:2}]
Let $\sigma$ be a permutation of $\{1,2,3\}$ for which $d_3 = d_{\sigma(1)}$, and \begin{equation*}
\Phi(d_{\sigma(1)}) = \Phi(d_3) <|\Phi(d_1) - \Phi(d_2)| = \Phi(d_{\sigma(3)}) - \Phi(d_{\sigma(2)}).
\end{equation*}
This implies that \begin{equation} \label{this:eq10:lem14}
\Phi(d_{\sigma(3)}) > \Phi(d_{\sigma(1)}) + \Phi(d_{\sigma(2)}).
\end{equation}
Note that \eqref{this:eq10:lem14} is precisely the case of Claim \ref{this:claim:1.3:1}, implying that $(d_{\sigma(1)},d_{\sigma(2)},d_{\sigma(3)})$ is not an $S$-triangle. By Remark \ref{this:permutation:remark}, the triple $(d_1,d_2,d_3)$ is not an $S$-triangle either, proving our claim. %
\end{proof}
The lemma now follows from Claims \ref{this:claim:1.3:1} and \ref{this:claim:1.3:2}.
\end{proof}

We are now ready to prove Theorem~\ref{this:delta:thm}.
\begin{proof}[Proof of Theorem \ref{this:delta:thm}] Let $\epsilon>0$ and let $\Delta$ be a distance value set with $\inf \Delta=0$. Let $(X, d)$ be a finite metric space satisfying $\diam (X, d)\leq\sup \Delta$.

Let $S$ be the set of all nonzero distance vales of $(X, d)$. The theorem is trivial if $S\subseteq \Delta$. We assume $S\setminus \Delta\neq \varnothing$. Let $t_i$, $1\leq i \leq m$, be an increasing enumeration of the elements of $S \setminus \Delta$. Define $\Gamma_S$ as in \eqref{set:gamma:fors} and let $\gamma = \min \Gamma_S$. Let
$$ \delta=\min\{|s-t|\colon s, t\in S\mbox{ and } s\neq t\}. $$ 
Then $\gamma,\delta>0$. If $\sup\Delta$ is finite, then by Definition~\ref{def:dvs}, $\sup\Delta\in \Delta$; in this case $t_m<\sup\Delta$. In this case, we also set $\beta=\sup\Delta-t_m$. Let $\beta=1/2$ if $\sup\Delta=+\infty$. Finally let
$$ \epsilon' = \displaystyle\frac{1}{2}\min\left\{\gamma, \delta, \beta, \epsilon\right\}.$$ 
Since $\inf\Delta=0$, $\Delta$ is dense in the interval $(0, \sup\Delta)$. Hence for each $1\leq i\leq m$, we can find $t'_i\in \Delta$ such that 
$$ \displaystyle\frac{\epsilon'}{2^{i+1}}<t'_i-t_i<\frac{\epsilon'}{2^i}. $$
Since $\epsilon'<\delta$, we have that $t'_i\not\in S$ for all $1\leq i\leq m$.

Let $S'=(S\cap \Delta)\cup\{t_i'\colon 1\leq i\leq m\}$ and define $\Phi\colon S\to S'$ by letting $\Phi(s)=s$ if $s\in S\cap \Delta$ and $\Phi(t_i)=t_i'$ for $1\leq i\leq m$. Define a distance function $d'$ on $X$ by letting, for any $x, y\in X$,
$$ d'(x,y)=\left\{\begin{array}{ll} d(x,y), & \mbox{ if $d(x,y)\in \Delta$,} \\ \Phi(d(x,y)), & \mbox{ otherwise.} \end{array}\right. $$
Since $\Delta$ is a distance value set, $S$ and $S'$ satisfy the hypothesis of Lemma~\ref{this:triangular:lemma}. It is also clear that $\Phi$ satisfies the hypothesis of Lemma~\ref{this:triangular:lemma}. It now follows from Lemma~\ref{this:triangular:lemma} that $d'$ is a metric on $X$. Now conditions (a)--(c) of Definition~\ref{def:remet} are obvious with this definition of $d'$.

We only verify (d). For this suppose $\varphi\in\Iso(X,d)$ first. Then for any $x, y\in X$, $d(\varphi(x), \varphi(y))=d(x,y)$. Then by (c), $d'(\varphi(x), \varphi(y))=d'(x,y)$. Thus $\varphi\in \Iso(X,d')$. Conversely, suppose $\varphi\in\Iso(X,d')$ and $x, y\in X$. Then $d'(\varphi(x), \varphi(y))=d'(x,y)$. If $d'(x,y)=d(x,y)$, then $d(x,y)\in \Delta$ and 
$$d(\varphi(x), \varphi(y))=d'(\varphi(x), \varphi(y))=d'(x,y)=d(x,y). $$
If $d'(x,y)\neq d(x,y)$, then $d(x,y)\not\in \Delta$ and we must have $d'(x,y)=\Phi(d(x,y))$ and $d'(\varphi(x), \varphi(y))=\Phi(d(\varphi(x), \varphi(y))$. Since $\Phi$ is a bijection, we conclude again $d(\varphi(x), \varphi(y))=d(x,y)$. Thus $\varphi\in\Iso(X,d)$.
\end{proof}

\subsection{Concentration functions on finite powers} \label{subsection:martingale}
In this subsection we deal with finite powers of finite $\Delta$-metric spaces. In previous research, the finite powers are given the normalized Hamming metric, which is not necessarily a $\Delta$-metric. Here we remetrize the finite powers using Theorem~\ref{this:delta:thm}, but need to verify that the concentration of measure still takes place with the $\Delta$-metric approximations. The goal of this subsection is to give an estimate of the concentration function of the approximate $\Delta$-metric space which will then be used in the proofs in the rest of the paper. We do this by applying the martingale technique of \cite[Section 4.3]{Pestov2}, in particular \cite[Theorem 4.3.18]{Pestov2}, which boils down to a computation of the so-called \emph{length} of a metric-measure space.

We use the following notation. Given a set $X$ and its finite power $X^m$, for a given element $x \in X^m$ we shall write $x_i$, for $1\leq i \leq m$, to refer to its $i$-th coordinate. If $G$ is a subset of the group $\Sym(X)$ of all permutations of $X$, then for $g \in G^m$ and $x \in X^m$, by $g(x)$ we mean the element of $X^m$ where $g(x)_i=g_i(x_i)$ for all $1\leq i\leq m$, that is
$$\begin{array}{rcl}g(x) &= &(g(x)_1,g(x)_2,\dots,g(x)_{m-1},g(x)_m) \\
&= &(g_1(x_1),g_2(x_2),\dots,g_{m-1}(x_{m-1}),g_m(x_m)).
\end{array}$$
For each $i = 1,\dots,m$ and $x\in X^m$, define \begin{equation}
\Omega_i(x) = \left\{y \in X^m\colon x_j = y_j \text{ for all } j \leq i \right\}.
\end{equation}
Also let $\Omega_0(x)=X$. Then
$$ X=\Omega_0(x)\supseteq\Omega_1(x)\supseteq\Omega_2(x)\supseteq \cdots \supseteq \Omega_{m-1}(x)\supseteq \Omega_m(x)=\{x\}. $$
For any $1\leq i\leq m$, let
$$ \Omega_i=\{\Omega_i(x)\colon x\in X^m\}. $$
Then each $\Omega_i$ is a partition of $X^m$, $\Omega_0=\{X\}$, $\Omega_m=\{\{x\}\colon x\in X^m\}$ consists of all singletons in $X^m$, and each $\Omega_{i+1}$ is a refinement of $\Omega_i$ (we denote $\Omega_i\prec \Omega_{i+1}$). In summary, \begin{equation} \label{partition:chain}
\{X\} = \Omega_0 \prec \Omega_1 \prec \Omega_2 \prec \dots \prec \Omega_{m-1} \prec \Omega_{m} = \{ \{x\}: x \in X^m \}.
\end{equation}

Given a finite metric space $(X,d)$ and its finite power $X^m$ we denote by $d_m$ the normalized Hamming metric on $X^m$, that is, for $x, y\in X^m$, we have \begin{equation}
d_m(x,y) = \frac{1}{m} \sum_{j = 1}^m d(x_i,y_i). 
\end{equation}

Let $\Delta$ be a countable distance value set with $\inf\Delta=0$. Let $(X,d)$ be a finite $\Delta$-metric space. For each integer $m \in \N^+$, the product space $X^m$ equipped with $d_m$ satisfies that $\diam(X^m, d_m) \leq \sup \Delta$. Therefore, for each $\epsilon > 0$ we may apply Theorem \ref{this:delta:thm} to find a $(\Delta,\epsilon)$-remetrization $D=d'_m$ for $X^m$.

\begin{definition} \label{this:sup:metric:def}
Let $m \in \N^+$, and assume $(X^m,D)$ is a power of a finite metric space $X$ equipped with a metric $D$. Assume $G$ is a set of permutations of $X$, and $G^m$ is its power. We define the metric $\delta_D$ on $G^m$ as follows: \begin{equation} \label{this:eq:mar20:2024}
\delta_D(g,h) = \sup_{x \in X^m} D(g^{-1}(x),h^{-1}(x))
\end{equation}
for any $g, h\in G^m$.
\end{definition}

The slightly awkward inverses are used to make the metric $\delta_D$ right invariant when $G$ is closed under composition, i.e., for any $f, g, h\in G^m$, we have $\delta_D(gf, hf)=\delta_D(g, h)$. The right-invariance is not needed until Section~\ref{first:density:section}.

\begin{lemma} \label{this:lemma:length}
Let $\Delta$ be a countable distance value set with $\inf\Delta=0$, and let $\epsilon > 0$ and $m \in \N^+$ be arbitrary. Let $(X,d)$ be a finite $\Delta$-metric space, let $a=\diam(X,d)$, and let $G$ be a set of isometries of $(X,d)$. Let $D=d_m'$ be a $(\Delta, \epsilon/m)$-remetrization of $(X^m,d_m)$. Let $\{\Omega_i\colon 1\leq i \leq m\}$ be the collection of partitions of $G^m$ as in \eqref{partition:chain}. Fix $1\leq i\leq m$. Let $f \in G^m$ and $g,h \in \Omega_{i-1}(f)$ be arbitrary. Define a map $\phi: \Omega_{i}(g) \to \Omega_{i}(h)$ as follows: \begin{equation}
\phi(w)_j = \left\{\begin{array}{ll}w_j, & \mbox{ if $j < i$,} \\
h_i, & \mbox{ if $j = i$, }\\
w_j, & \mbox{ if $j>i$,}
\end{array}\right.
\end{equation}
for each $w \in \Omega_{i}(g)$. If $G^m$ is equipped with the metric $\delta_D$ as in \eqref{this:eq:mar20:2024} with respect to $D$, then the map $\phi$ is an isometry between the subspaces $\Omega_{i}(g)$ and $\Omega_{i}(h)$ of $(G^m, \delta_D)$. Moreover, for each $w \in \Omega_{i}(g)$, we have \begin{equation} \label{this:length:equation}
{\delta_D}(w,\phi(w)) \leq \frac{a + \epsilon}{m}.
\end{equation}
\end{lemma}

\begin{proof}
Note that the map $\phi$ is well defined: indeed, for each $j < i$ we have $w_j = g_j= f_j = h_j$, and $w_i = h_i$ by construction. The map is also injective: given $v,w \in \Omega_i(g)$ we have that $v \neq w$ if and only if there exists $j > i$ satisfying $v_j \neq w_j$; in this case, $\phi(v)_j = v_j \neq w_j = \phi(w)_j$ holds, proving injectivity. Surjectivity is clear from the definition of $\phi$. %

By definition, for any $v, w\in \Omega_i(g)$ and $x\in X^m$,
$$\begin{array}{rcl} d_m(v^{-1}(x), w^{-1}(x))&=&\displaystyle\frac{1}{m}\sum_{j>i} d(v_i^{-1}(x_i),w_i^{-1}(x_i)) \\ \\
&=& d_m(\phi(v)^{-1}(x), \phi(w)^{-1}(x)). 
\end{array}$$
By Theorem~\ref{this:delta:thm} and Definition~\ref{def:remet} (c), we have
$$ D(v^{-1}(x), w^{-1}(x))=D(\phi(v)^{-1}(x), \phi(w)^{-1}(x)). $$
It follows that $\phi$ is an isometry between $\Omega_i(g)$ and $\Omega_i(h)$ as subspaces of $(G^m, \delta_D)$. 

Let us now verify \eqref{this:length:equation}.
Note that, for any $w \in \Omega_i(g)$ and $x \in X^m$, we have\begin{equation*}
\begin{array}{rcl} d_m\left( w^{-1}(x),\phi(w)^{-1}(x) \right) &=& %
 \displaystyle\frac{1}{m} d \left( w^{-1}(x)_i, \phi(w)^{-1}(x)_i \right) \\ \\
 &=& \displaystyle\frac{1}{m} d(g_i^{-1}(x_i), h_i^{-1}(x_i))\leq\displaystyle\frac{a}{m}.\end{array}
\end{equation*}
By Theorem~\ref{this:delta:thm} and Definition~\ref{def:remet} (b), we have
\begin{equation*}
\left| D \left(w^{-1}(x),\phi(w)^{-1}(x) \right) - d_m\left(w^{-1}(x),\phi(w)^{-1}(x) \right)\right| < \frac{\epsilon}{m}  .
\end{equation*}
Therefore, \begin{equation*}
\begin{array}{rcl}\delta_D(w, \phi(w))&=&\displaystyle\sup_{x \in X^m} D\left(w^{-1}(x),\phi(w)^{-1}(x) \right) \\
&\leq & \displaystyle \sup_{x \in X^m} d_m\left(w^{-1}(x),\phi(w)^{-1}(x)  \right) + \displaystyle\frac{\epsilon}{m} \leq \displaystyle\frac{a + \epsilon}{m}\end{array}
\end{equation*}
as desired.
\end{proof}

Any finite metric space becomes a metric-measure space with the normalized counting measure. We can now apply the martingale technique (\cite[Theorem 4.3.18]{Pestov2}) to obtain an estimate of the concentration function (see Definition~\ref{def:confun}) of $(G^m, \delta_D)$.

\begin{theorem} \label{this:concentration:thm}
Let $\Delta$ be a countable distance value set with $\inf\Delta=0$, and let $\epsilon > 0$ and $m \in \N^+$ be arbitrary. Let $X$ be a finite $\Delta$-metric space, let $a=\diam(X,d)$, and let $G$ be a set of isometries of $(X,d)$. Let $D=d_m'$ be a $(\Delta,\epsilon/m)$-remetrization of $(X^m, d_m)$. Then the concentration function $\alpha_{G^m}$ of $(G^m,\delta_D)$ satisfies that 
\begin{equation} \label{this:concentration:estimation} 
\alpha_{G^m}(\rho) \leq 2e^{\frac{-m \rho^2}{16(a+\epsilon)^2}}.
\end{equation}
\end{theorem}

\begin{proof}
According to \cite[Definition 4.3.16]{Pestov2}, Lemma \ref{this:lemma:length} yields an estimation for the so-called length $\ell$ of $G^m$ as follows:\begin{equation*}
\ell^{2} \leq \sum_{i=1}^m \left( \frac{a+\epsilon}{m} \right)^2 = \frac{(a+\epsilon)^2}{m}.
\end{equation*}
We then apply \cite[Theorem 4.3.18]{Pestov2} to deduce \eqref{this:concentration:estimation}. 
\end{proof}

\section{The L\'evy Property of $\Iso(\U_\Delta)$} \label{first:density:section}

In this section we prove that the isometry group $\Iso(\U_\Delta)$ has the strong L\'evy property for any countable distance value set $\Delta$ with $\inf\Delta=0$. Recall again that the $\Iso(\U_\Delta)$ is equipped with the pointwise convergence topology, where the underlying space $\U_\Delta$ is equipped with the metric topology. 

A special case of our theorem is when $\Delta$ is the set of all positive rationals, in which case $\U_\Delta$ is customarily denoted $\mathbb{QU}$ and known as the \emph{rational Urysohn metric space}. Since $\mathbb{QU}$ is a dense subspace of $\U$ and $\Iso(\mathbb{QU})$ is a dense topological subgroup of $\Iso(\U)$, our result implies Pestov's theorem \cite{Pestov1} that $\Iso(\U)$ has the strong L\'evy property. 

Our proof is different from Pestov's proof in \cite{Pestov1}. Instead of constructing the witnessing sequence in $\Iso(\U_\Delta)$, we consider the Fra\"iss\'e class of actions of finite isometry groups on finite $\Delta$-metric spaces (as in \cite{Gao-Adv}). In this setting a witnessing sequence naturally emerges from the canonical construction of the Fra\"iss\'e limit. 

Our proof also shows that the finite groups in a witnessing sequence for the strong L\'evy property of $\Iso(\U_\Delta)$ union up to a countable locally finite group which is isomorphic to Hall's group $\H$, thus in particular this gives a topology on $\H$ to make it a L\'evy group. By our results in Sections \ref{section:examples} and \ref{section:newcountable}, $\H$ is thus a new countable group which admits a topology with the L\'evy property.

We remark that the notation $\Iso(\U_\Delta)$ has sometimes been used in the literature (including \cite{Gao-Adv}) to denote the automorphism group of $\U_\Delta$ as a countable metric structure, in which case the topology on it is the pointwise convergence topology, where the underlying space $\U_\Delta$ is equipped with the discrete topology. To avoid confusion, in this paper we denote this group by $\Aut(\U_\Delta)$. Thus $\Aut(\U_\Delta)$ and $\Iso(\U_\Delta)$ are the same as sets, but they carry different topologies. $\Aut(\U_\Delta)$ is essentially an automorphism group of a countable structure, and cannot have the L\'evy property by Corollary~\ref{cor:nonLevy}.

We recall some basic results proved in \cite{Gao-Adv}.

\begin{definition}[{\cite[Definition 3.8]{Gao-Adv}}]
Given a countable distance value set $\Delta$, let $\mathcal{K}_\Delta$ denote the set of all structures $(X,G)$ where $X$ is a finite $\Delta$-metric space, and $G$ is a finite group acting isometrically on $X$.
\end{definition}

Given two pairs $(X_1,G_1), (X_2,G_2) \in \mathcal{K}_\Delta$,  an \emph{embedding} from $(X_1,G_1)$ to $(X_2,G_2)$ is a pair $(\phi,\pi)$, where $\phi: G_1 \to G_2$ is a group embedding and $\pi: X_1 \to X_2$ is an isometric embedding, such that $\phi(g)(\pi(x)) = \pi(g(x))$ holds for each $g \in G_1$ and $x \in X_1$. With this concept in place, one can consider the defining properties of a Fra\"iss\'e class, and it was shown in \cite[Theorem 3.9]{Gao-Adv} that $\mathcal{K}_\Delta$ is a Fra\"iss\'e class.
If we denote by $(X_\Delta,H_\Delta)$ the Fra\"iss\'e limit of $\mathcal{K}_\Delta$, then $X_\Delta$ is isometric to $\U_\Delta$ and $H_\Delta$ is isomorphic to Hall's group $\H$ (see {\cite[Lemmas 3.10 and 3.12]{Gao-Adv}}). It was also shown in \cite[Lemma 3.13]{Gao-Adv} that $H_\Delta$ is a dense topological subgroup of $\Aut(X_\Delta)$; however, since $\Aut(X_\Delta)$ and $\Iso(X_\Delta)$ are the same sets but the topology of $\Aut(X_\Delta)$ is finer than that of $\Iso(X_\Delta)$, $H_\Delta$ continues to be dense in $\Iso(X_\Delta)$.

\begin{definition}\label{def:62}
For each $(X, G)\in \mathcal{K}_\Delta$, let $d_X$ be the $\Delta$-metric on $X$, and define
$$ \delta_G(g, h)=\sup_{x\in X}d_X(g^{-1}(x), h^{-1}(x)) $$
for $g, h\in G$. Then $\delta_G$ is a right-invariant $\Delta$-metric on $G$. 
\end{definition}

\begin{definition}\label{def:sup-inverse}
We say that a subset $\mathcal{D} \subseteq \mathcal{K}_\Delta$ is \emph{cofinal} in $\mathcal{K}_\Delta$ if for any structure $(X,G) \in \mathcal{K}_\Delta$ there exist a structure $(Y,H) \in \mathcal{D}$ and an embedding from $(X, G)$ to $(Y,H)$. 
\end{definition}

\begin{definition}
Given a positive integer $n \in \N^+$, we shall use $\beta_{n}$ to denote the function $\beta_{n}: [0,\infty) \to [0,\infty)$ defined by \begin{equation}
\beta_{n}(\varepsilon) = 2 e^{-n \varepsilon^2}.
\end{equation}
\end{definition}

Given two functions $\beta,\gamma: [0,\infty) \to [0,\infty)$ we shall write $\gamma \leq \beta$ if and only if $\gamma(t) \leq \beta(t)$ holds for all $t \in [0,\infty)$. %

\begin{lemma} \label{this:density:lemma}
Let $\Delta$ be a countable distance value set with $\inf\Delta=0$. For each $n \in \N^+$, the set \begin{equation}
\mathcal{D}_{n} = \left\{(X,G) \in \mathcal{K}_\Delta\colon \alpha_{G} \leq \beta_n \right\}
\end{equation}
is cofinal in $\mathcal{K}_\Delta$, where $\alpha_{G}$ is the concentration function of $(G, \delta_G, \mu)$ and $\mu$ is the normalized counting measure on $G$. 
\end{lemma}

\begin{proof}
Let $n \in \N^+$ and $(X,G) \in \mathcal{K}_\Delta$ be arbitrary, and set $a=\diam(X,d)$. Let $m \in \N^+$ be a positive integer satisfying \begin{equation}
m \geq 16(a + 2^{-n})^2 \cdot n. 
\end{equation}
Let $d_m$ be the normalized Hamming metric on $X^m$ and let $D=d'_m$ be a $(\Delta,2^{-n}/m)$-remetrization for $(X^m, d_m)$. Consider $(X^m, G^m)\in \mathcal{K}_\Delta$, where $X^m$ is  equipped  with the $\Delta$-metric $D$, and the action of $G^m$ on $X^m$ is given by 
$$\gamma\cdot y=\gamma(y)=(\gamma_1(y_1),\dots, \gamma_m(y_m))$$ 
for $\gamma\in G^m$ and $y\in X^m$. Let $\pi\colon (X, d)\to (X^m, D)$ be defined as
$$ \pi(x)=(x, \dots, x)\in X^m. $$
Then $\pi$ is an isometric embedding by Definition~\ref{def:remet} (c). Let $\phi\colon G\to G^m$ be defined as
$$ \phi(g)=(g,\dots, g)\in G^m. $$
It is easily seen that $\phi(g)(\pi(x))=\pi(g(x))$ for all $g\in G$ and $x\in X$. Thus $(\phi, \pi)$ is an embedding from $(X, G)$ to $(X^m, G^m)$, 
Moreover, for any $\gamma, \eta\in G^m$, by Definitions~\ref{this:sup:metric:def} and \ref{def:62}, we have
$$\delta_D(\gamma, \eta)=\sup_{y\in X^m}D(\gamma^{-1}(y), \eta^{-1}(y))=\delta_{G^m}(\gamma, \eta).$$

By Theorem~\ref{this:concentration:thm},  we have \begin{equation*}
\alpha_{G^m}(\varepsilon) \leq 2e^ \frac{-m \varepsilon^2}{16(a + 2^{-n})^2} \leq 2 e^{-n \varepsilon^2} = \beta_n(\varepsilon).
\end{equation*}
This shows that $(X^m,G^m) \in \mathcal{D}_n$. 
\end{proof}

\begin{theorem}\label{thm:UD1} Let $\Delta$ be a countable distance value set with $\inf\Delta=0$. Then $H_\Delta$ equipped with the subspace topology of $\Iso(\U_\Delta)$ has the strong L\'evy property.
\end{theorem}

\begin{proof} By the canonical construction of the Fra\"iss\'e limit of a Fra\"iss\'e class, we obtain a direct system $\{(X_n, G_n), (\phi_n, \pi_n)\colon n\in \mathbb{N}^+\}$ so that for all $n\in\N^+$, $(X_n,G_n)\in \mathcal{D}_n$, the direct limit of $X_n$ is $X_\Delta$, and the direct limit of $G_n$ is $H_\Delta$. Since $X_\Delta$ is isometric to $\U_\Delta$, we may assume $X_\Delta=\U_\Delta$ and the metric on it is $d_{\U_\Delta}$. For notational simplicity, we also assume $\phi_n$ and $\pi_n$ are all identity maps. Thus each $X_n$ is a finite subspace of $\U_\Delta$, and each $G_n$ is a finite subgroup of $H_\Delta$, the latter being a dense subgroup of $\Iso(\U_\Delta)$.

To verify the (strong) L\'evy property of $H_\Delta$, take an open neighborhood $V$ of the identity of $H_\Delta$. Since $H_\Delta$ has the subspace topology of $\Iso(\U_\Delta)$, whose basic open neighborhoods are given by \eqref{eq:nbhd}, we may assume that $V=V[x_1,\dots, x_k;\varepsilon]\cap H_\Delta$ for some $x_1, \dots, x_k\in \U_\Delta$ and $\varepsilon>0$. Let $A_n\subseteq G_n$ be subsets with $\mu_n(A_n)\geq 1/2$, where $\mu_n$ is the normalized counting measure on $G_n$. 
Let
$$ V_n(\varepsilon)=\left\{ g\in G_n\colon d_{X_n}(g(x), x)<\varepsilon \mbox{ for all $x\in X_n$}\right\}. $$
Let $N$ be large enough such that $x_1,\dots, x_k\in X_N$. Then by definition, for all $n\geq N$, we have
$$        V_n(\varepsilon)A_n\subseteq G_n\cap VA_n=G_n\cap V[x_1,x_2,...,x_k;\varepsilon]A_n. $$
Note that
$$ V_n(\varepsilon)A_n=\left\{g\in G_n\colon \delta_{G_n}(g, f)<\varepsilon \mbox{ for some $f\in A_n$}\right\}. $$
In fact, if $g\in V_n(\varepsilon)A_n$, then for some $f\in A_n$, $gf^{-1}\in V_n(\varepsilon)$; by the definition of $\delta_{G_n}$, we have $\delta_{G_n}(fg^{-1}, e)<\varepsilon$, and by the right-invariance of $\delta_{G_n}$, we have $\delta_{G_n}(f, g)<\varepsilon$. The argument for the converse is similar.

Now for all $n\geq N$, since $(X_n, G_n)\in \mathcal{D}_n$, we have
    \begin{align*}
        1-\mu_n(V_n(\varepsilon)A_n)\leq \alpha_{G_n}(\varepsilon)\leq \beta_n(\varepsilon)\leq 2e^{-n\varepsilon^2}.
    \end{align*}
    i.e.
    \begin{align*}
        \mu_n(V_n(\varepsilon)A_n)\geq 1-2e^{-n\varepsilon^2}.
    \end{align*}
    Thus 
    \begin{align*}
     1\geq   \mu_n(VA_n)\geq \mu_n(V_n(\varepsilon)A_n)\geq 1-2e^{-n\varepsilon^2}\to 1
    \end{align*}
    as $n\to \infty$. Thus $H_\Delta$ is a strong L\'evy group.
\end{proof}

Since $H_\Delta$ is isomorphic to $\H$ and is dense in $\Iso(\U_\Delta)$, we have the following immediate corollary.

\begin{corollary}\label{cor:main} Let $\Delta$ be a countable distance value set with $\inf\Delta=0$. Then $\Iso(\U_\Delta)$ has the strong L\'evy property. Moreover, there is a topology $\tau$ on $\H$ such that $(\H, \tau)$ has the L\'evy property.
\end{corollary}

\section{The L\'evy Property of Omnigenous Groups} \label{section:omnigen:construction}

In this section, we will give another proof of Corollary~\ref{cor:main}. This proof is more similar to Pestov's proof of the strong L\'evy property of $\Iso(\U)$. Namely, we will construct the witnessing sequence by a recursive construction within $\Iso(\U_\Delta)$. This approach allows us to also generalize the result from $\H$ to arbitrary countable omnigenous locally finite groups. By results of \cite{Gao-Adv} and \cite{Gao-Li}, there are many distinct isomorphism types of such groups (see Section~\ref{section:complexity} for more details), therefore we obtain many more examples of countable groups with the L\'evy property.

We first recall some basic concepts and results from \cite{Gao-Adv}.

\begin{definition}[{\cite[Definition 5.2]{Gao-Adv}}] \label{omnigen:def}
    Let $G$ be a group. We say that $G$ is \emph{omnigenous} if for every finite group $G_1\leq G$, finite groups $\Gamma_1\leq\Gamma_2$ and group isomorphism $\Psi_1\colon G_1\cong\Gamma_1$, there is a finite subgroup $G_2\leq G$ with $G_1\leq G_2$ and an onto homomorphism $\Psi_2\colon G_2\rightarrow\Gamma_2$ such that $\Psi_2\restriction G_1=\Psi_1$.
\end{definition}

Note that in Definition~\ref{omnigen:def}, if we replace the onto homomorphism $\Psi_2$ by an isomorphism, then we obtain a defining property (property (iii)) of $\H$, that is, up to isomorphism, $\H$ is the unique countable locally finite group with this property. In particular, $\H$ is omnigenous. 

A notion of $\infty$-MIF was also introduced in \cite[Definition 6.17]{Gao-Adv}, and it was proved that for a locally finite group, $\infty$-MIF implies MIF, and $\infty$-MIF is equivalent to omnigenousness (\cite[Proposition 6.18]{Gao-Adv}). In particular, all omnigenous locally finite groups are MIF. It follows, for example, the examples in Examples~\ref{ex:1}, \ref{ex:2} and \ref{ex:4} are all not  omnigenous.  We do not know, however, if the examples in Example~\ref{ex:5} are omnigenous.

Our recursive construction of the witnessing sequence of the strong L\'evy property of $\Iso(\U_\Delta)$ follows the general outline of the proof of \cite[Theorem 5.3]{Gao-Adv}, where it was shown that any omnigenous group embeds into $\Aut(\U_\Delta)$ as a dense subgroup. In particular, we will use the following concepts and results.

\begin{definition}[{\cite[Definition 2.9]{Gao-Adv}}]
    Let $X$ be a metric space, and let $\mathcal{P}_X$ be the space of all partial isometries of $X$, i.e. isometries between finite subspaces of $X$. Let $P\subseteq\mathcal{P}_X$ be such that $P=P^{-1}$. An \emph{S-extension} of $X$ with respect to $P$ is a pair $(Y,\phi)$, where $Y\supseteq X$ is a metric space extending $X$, and $\phi\colon P\rightarrow \Iso(Y)$ satisfies that $\phi(p)$ extends $p$ for all $p\in P$. We also require $\phi(p^{-1})=\phi(p)^{-1}$ for all $p\in P$. When $P=\mathcal{P}_X$ we call $(Y,\phi)$ an \emph{S-extension} of $X$.
\end{definition}

\begin{definition}[{\cite[Definition 2.10]{Gao-Adv}}]
    Let $X$ be a metric space. An S-extension $(Y,\phi)$ of $X$ is \emph{strongly coherent} if for every triple $(p,q,r)$ of partial isometries of $X$ such that $p\circ q=r$, we have $\phi(p)\circ\phi(q)=\phi(r)$.
\end{definition}
\begin{theorem}
[Solecki {\cite{Rosendal,Sin-Sol,Solecki}}]    
\label{this:s-extension:lemma1}
    Let $\Delta$ be any distance value set and let $X$ be a finite $\Delta$-metric space. Then $X$ has a finite, strongly coherent S-extension $(Y,\phi)$ where $Y$ is a $\Delta\sh$metric space.
\end{theorem}

\begin{lemma}[{\cite[Lemma 2.12]{Gao-Adv}}]\label{this:s-extension:lemma2}
    Let $X$ be a metric space and $(Y,\phi)$ be a strongly coherent S-extension of $X$. For every $D\subseteq X$, the map $p\mapsto\phi(p)$ gives a group embedding from $\Iso(D)$ to $\Iso(Y)$.
\end{lemma}

\begin{lemma}[{\cite[Lemma 5.1]{Gao-Adv}}]\label{this:omnigenous:lemma1}
    Let $\Delta$ be any countable distance value set. Let $X$ be a finite $\Delta\sh$metric space. Let $\Lambda\leq\Gamma$ be finite groups and $\pi:\Lambda\rightarrow\Iso(X)$ be an isomorphic embedding. Then there is a finite $\Delta\sh$metric space $Y$ extending $X$ and an isomorphic embedding $\pi':\Gamma\rightarrow\Iso(Y)$ such that for any $\gamma\in\Lambda$ and $x\in X$, $\pi'(\gamma)(x)=\pi(\gamma)(x)$.
\end{lemma}

Note that, in the proof of the original \cite[Theorem 5.3]{Gao-Adv}, metrics on finite groups are irrelevant. Here, however, metrics on these finite groups play an important role as we unravel the concentration-of-measure phenomenon. For a finite metric space $X$ and a subgroup $G$ of $\Iso(X)$, we always use the sup-inverse metric $\delta_G$ of Definition~\ref{def:sup-inverse}. Recall that $\delta_G$ is a right-invariant $\Delta$-metric if $X$ is a $\Delta$-metric space. Moreover, if $X$ is a finite metric space and $\pi\colon G\to \Iso(X)$ is an embedding of $G$ into $\Iso(X)$, we regard $G$ as a subgroup of $\Iso(X)$ and equipped it with the metric $\delta_G=\delta_{\pi(G)}$ accordingly.

The following is our key technical lemma.
\begin{lemma}\label{this:omnigenous:lemma2}
Let $\Delta$ be a countable distance value set with $\inf\Delta=0$, and $X\subseteq Y$ be finite $\Delta$-metric spaces. Let $G_1\leq G_2$ be finite subgroups, and let $\pi_1\colon G_1\rightarrow \Iso(X)$ and $\pi_2\colon G_2\rightarrow\Iso(Y)$ be group embeddings such that for all $g\in G_1$, $\pi_2(g)$ extends $\pi_1(g)$. Let $H$ be a countable omnigenous locally finite group. Suppose $G_1\leq H$. Then there exist a finite subgroup $G_3\leq H$, a finite $\Delta$-metric space $Z$ and an embedding $\pi_3\colon G_3\rightarrow\Iso(Z)$ such that:
    \begin{itemize}
        \item $G_1\leq G_3$;
        \item $Y\subseteq Z$ as metric spaces;
        \item for any $g\in G_1$, $\pi_3(g)$ extends $\pi_1(g)$;
        \item for any $g\in G_2$ there is $h\in G_3$ such that $\pi_3(h)$ extends $\pi_2(g)$; and
        \item the concentration functions satisfy $\alpha_{G_3}\leq\alpha_{G_2}$.
    \end{itemize}
\end{lemma}
\begin{proof}
    Since $H$ is omnigenous, there is a finite subgroup $G_3\leq G$ such that $G_1\leq G_3$ and there is an onto homomorphism $\Psi:G_3\rightarrow G_2$ that extends the isomorphism between copies of $G_1$ in $G_3$ and in $G_2$. %
    Let $a=\diam(Y,d_Y)$ and choose $b\in\Delta$ small enough such that 
    $$b<\displaystyle\frac{1}{3}\min\left\{r>0\colon r=d_Y(y_1, y_2) \mbox{ for some $y_1, y_2\in Y$}\right\}. $$ Now consider the disjoint union $Z=Y\sqcup G_3$. Define a metric $d_Z$ on $Z$ by 
    \begin{align*}
        d_Z(z_1,z_2)=\left\{\begin{array}{ll}
            d_Y(z_1,z_2), & \mbox{if $z_1,z_2\in Y$},\\
           b, & \mbox{if $z_1,z_2\in G_3$,}\\
            a, & \mbox{if $z_1\in Y$, $z_2\in G_3$ or $z_2\in Y$, $z_1\in G_3$.}
        \end{array}\right.
    \end{align*}
for $z_1,z_2\in Z$.    It is easily seen that $(Z,d_Z)$ is a $\Delta$-metric space. Next, we construct an embedding $\pi_3\colon G_3\to \Iso(Z)$. For any $g\in G_3$, define
    \begin{align*}
        \pi_3(g)(x)=\left\{\begin{array}{ll}
            \pi_2(\Psi(g))(x), & \mbox{ if $x\in Y$},\\
            gx, & \mbox{ if $x\in G_3$}.
        \end{array}\right.
    \end{align*}
Then it is easily seen that $\pi_3$ satisfies all the conditions except perhaps the last one about the concentration functions. %
Before considering the concentration functions, observe the following easy facts:
\begin{enumerate}
\item[(i)] for any $\gamma, \eta\in G_2$, we have that $\delta_{G_2}(\gamma, \eta)<3b$ if and only if $\delta_{G_2}(\gamma,\eta)=0$; 
\item[(ii)] for any $g, h\in G_3$, if $\Psi(g)=\Psi(h)$, then $\delta_{G_3}(g, h)\leq b$; 
\item[(iii)] for any $A\subseteq G_3$, $\mu_{G_3}(A)\leq \mu_{G_2}(\Psi(A))$;
\item[(iv)] for any $B\subseteq G_2$, $\mu_{G_2}(B)=\mu_{G_3}(\Psi^{-1}(B))$.
\end{enumerate}
It follows from (i) that $\alpha_{G_2}(\varepsilon)=1/2$ for $\varepsilon\leq 3b$, and therefore $\alpha_{G_3}(\varepsilon)\leq 1/2=\alpha_{G_2}(\varepsilon)$ for $\varepsilon\leq 3b$. Now fix $\varepsilon>3b$. Let $A\subseteq G_3$ be such that $\mu_{G_3}(A)\geq 1/2$. Let
$$ A_{\varepsilon}= \left\{g\in G_3\colon \delta_{G_3}(g,h)<\varepsilon \mbox{ for some $h\in A$}\right\}. $$
Then by (iii), $\mu_{G_2}(\Psi(A))\geq \mu_{G_3}(A)\geq 1/2$. By (ii), we have that $\Psi^{-1}(\Psi(A))\subseteq A_\varepsilon$. It then follows from (iv) that
$$ 1-\mu_{G_3}(A_\varepsilon)\leq 1-\mu_{G_3}(\Psi^{-1}(\Psi(A)))=1-\mu_{G_2}(\Psi(A))\leq \alpha_{G_2}(\varepsilon). $$
Since $A$ is arbitrary, we get that $\alpha_{G_3}(\varepsilon)\leq \alpha_{G_2}(\varepsilon)$.
\end{proof}

We are now ready for the main theorem of this section, which is a strengthening of \cite[Theorem 5.3]{Gao-Adv}.

\begin{theorem} \label{this:main(og):thm}
Let $\Delta$ be a countable distance value set with $\inf \Delta = 0$, and let $H$ be a countable omnigenous locally finite group. Then there is a group embedding $\pi\colon H\to \Iso(\U_\Delta)$ with dense image such that $H$ equipped with the weakest topology to make $\pi$ continuous has the strong L\'evy property. 
\end{theorem}
\begin{proof}
    Since $\U_\Delta$ is countable, we may fix an enumeration $\{q_n\}_n$ of all partial isometries of $\U_\Delta$. Also fix an enumeration $\{h_n\}_n$ of $H$. We will inductively construct sequences of following objects:
    \begin{itemize}
    \item an increasing sequence $\{X_n\}_n$ of finite subsets of $\U_\Delta$;
    \item an increasing sequence $\{G_n\}_n$ of finite subgroups of $H$;
    \item group embeddings $\pi_n\colon G_n\rightarrow \Iso(X_n)$, for $n\geq 1$.
    \end{itemize}
They will satisfy the following properties for all $n\geq 1$:
\begin{enumerate}
\item[(1)] $X_n\subseteq X_{n+1}$  and $\U_\Delta=\bigcup_nX_n$;
\item[(2)] $h_n\in G_n$ and there exists $r_n\in G_n$ with $q_n\subseteq\pi_n(r_n)$;
\item[(3)] if $g\in G_n$ and $x\in X_n$, then $\pi_n(g)(x)=\pi_{n+1}(g)(x)$.
\end{enumerate}
Granting (1)--(3), we have $H=\bigcup_n G_n$ and we may define a group embedding $\pi\colon H\to \Iso(\U_\Delta)$ so that $\pi=\bigcup_n \pi_n$. The following will be satisfied in addition:
\begin{enumerate}
\item[(4)] the sequence $\{G_n\}_n$ of subgroups witnesses the L\'evy property of $\pi(H)$ equipped with the subspace topology of $\Iso(\U_\Delta)$.
\end{enumerate}
The conclusions of the theorem will follow from (2) and (4). 

Without loss of generality, assume $q_1$ is the empty function and $h_1=e$ is the identity of $H$. Let $X_1$ be an arbitrary singleton in $\U_\Delta$, let $G_1=\{e\}$, and let $\pi_1\colon G_1\to\Iso(X_1)$ be the obvious embedding. This takes care of the construction for $n=1$. Next, assume that $X_n$, $G_n$ and $\pi_n$ have been defined. We define $X_{n+1}$, $G_{n+1}$ and $\pi_{n+1}$. This will take several steps.

   Step 1: Let $H_{n+1}=\langle G_n, h_{n+1}\rangle_H$. By Lemma \ref{this:omnigenous:lemma1}, we obtain a finite $\Delta$-metric space $Y_{n+1}$ extending $X_n$ and a group embedding $\sigma_{n+1}\colon H_{n+1}\to\Iso(Y_{n+1})$ such that for any $g\in G_n$ and $x\in X_n$, $\sigma_{n+1}(g)(x)=\pi_n(g)(x)$. By the universality and the ultrahomogeneity of $\U_\Delta$, we may assume that $Y_{n+1}\subseteq \U_\Delta$.

   Step 2: Consider the finite $\Delta$-metric subspace of $\U_\Delta$ with domain 
   $$  W_{n+1}=Y_{n+1}\cup \dom(q_{n+1})\cup \rng(q_{n+1}). $$
  By Theorem \ref{this:s-extension:lemma1}, we obtain a finite, strongly coherent S-extension $(Z_{n+1},\phi)$ of $W_{n+1}$. So $\phi(q_{n+1})$ represents an element of $\Iso(Z_{n+1})$ which extends $q_{n+1}$. By Lemma \ref{this:s-extension:lemma2}, such $\phi$ gives rise to a group embedding from $\Iso(Y_{n+1})$ to $\Iso(Z_{n+1})$. Let us abuse the notation a little bit and still denote this embedding by $\phi\colon \Iso(Y_{n+1})\rightarrow\Iso(Z_{n+1})$. Let $\Psi_{n+1}=\phi\circ\sigma_{n+1}$. Let $\Lambda_1=\Psi_{n+1}(H_{n+1})\leq \Iso(Z_{n+1})$, and let $\Lambda_2=\langle\Lambda_1,\phi(q_{n+1})\rangle_{\Iso(Z_{n+1})}$. 
    
    Step 3: We now apply Lemma \ref{this:density:lemma} to the structure $(Z_{n+1},\Lambda_2)\in\mathcal{K}_\Delta$ to obtain a structure $(Q_{n+1}, \Lambda_3)\in \mathcal{D}_{n+1}$ and an embedding from $(Z_{n+1}, \Lambda_2)$ to $(Q_{n+1}, \Lambda_3)$. Hence, for all $\varepsilon>0$, we have
    $$\alpha_{\Lambda_3}(\varepsilon)\leq \beta_{n+1}(\varepsilon).$$

    Step 4: By Lemma \ref{this:omnigenous:lemma2}, we obtain $G_{n+1}\leq H$, a finite $\Delta$-metric space $X_{n+1}$ and an embedding $\pi_{n+1}\colon G_{n+1}\rightarrow\Iso(X_{n+1})$ such that:
    \begin{itemize}
        \item $H_{n+1}\leq G_{n+1}$;
        \item $Q_{n+1}\subseteq X_{n+1}$;
        \item for any $g\in H_{n+1}$, $\pi_{n+1}(g)$ extends $\pi_n(g)$;
        \item for any $g\in \Lambda_3\leq\Iso(Q_{n+1})$ there is $h\in G_{n+1}$ such that $\pi_{n+1}(h)$ extends $g$;
        \item $\alpha_{G_{n+1}}\leq\alpha_{\Lambda_3}\leq\beta_{n+1}$.
    \end{itemize}
    This finishes the construction of $X_{n+1}$, $G_{n+1}$ and $\pi_{n+1}$.
    
    It is straightforward to check that requirements (1)--(3) are fulfilled by this construction. The proof of (4) is similar to the proof of Theorem~\ref{thm:UD1}. 
 \end{proof}

Since $\H$ is omnigenous, Corollary~\ref{cor:main} is now an immediate corollary of Theorem~\ref{this:main(og):thm}.

\section{L\'evy Groups from Continuous Logic}\label{section:cont:logic}
In this section we use continuous logic to produce more examples of isometry groups with the strong L\'evy property. These groups are Polish. 

Much of our discussion in this section is based on concepts and results proved by Gao and Ren in \cite{Gao-Ren} regarding the existence of Urysohn continuous structures and the properties of their isometry groups. We first recall these concepts and results.

\begin{definition}[{\cite[Definition 2.1]{Gao-Ren}}]{\ }
\begin{enumerate}
   \item[(i)] A \emph{modulus of continuity} is a nondecreasing function
    \begin{align*}
        u\colon (0,\infty)\rightarrow(0,\infty]
    \end{align*}
    such that $\lim_{x\rightarrow 0}u(x)=0$.
    \item[(ii)] Let $u$ be a modulus of continuity. Let $(X_1,d_1)$ and $(X_2,d_2)$ be metric spaces. We say a mapping $f\colon X_1\rightarrow X_2$ is \emph{uniformly continuous with respect to $u$} if for all $x,y\in X_1$, we have
    \begin{align*}
        d_2(f(x),f(y))\leq u(d_1(x,y))
    \end{align*}
    \end{enumerate}
\end{definition}

\begin{definition}[{\cite[Definition 2.4]{Gao-Ren}}]
    In continuous logic, a \emph{continuous signature} $\Lan$ is a set consisting of a \emph{distance symbol} $d$, \emph{relation symbols} $(R_i)_{i\in I}$ and \emph{function symbols} $(f_j)_{j\in J}$. Each symbol of $\Lan$ is associated with two pieces of data: its arity, which is a natural number, and its modulus functions, where each coordinate is associated with a modulus function. We allow no $0$-ary relations, but $0$-ary functions are allowed (and they are understood as constants).
\end{definition}

\begin{definition}[{\cite[Definition 2.5]{Gao-Ren}}] Given a continuous signature $\Lan$,  an \emph{$\Lan$-pre-structure} $\mathcal{A}$ consists of a metric space $(A,d_A)$, along with
    \begin{itemize}
    \item $R^A\colon A^n\rightarrow [0,1]$ for each $n$-ary relation symbol $R\in \Lan$, and
    \item $f^A\colon A^n\rightarrow A$ for each $n$-ary function symbol $f\in \Lan$,
    \end{itemize}
   so that all of them are uniformly continuous with respect to the corresponding modulus functions.
   
    If $(A,d_A)$ is a complete metric space, then we say that $\mathcal{A}$ is an \emph{$\Lan$-structure}.
\end{definition}

    Substructures, isomorphisms, and isomorphic embeddings between continuous (pre-)structures can be naturally defined.
\begin{definition}[{\cite[Definition 4.1]{Gao-Ren}}]
        Let $\Lan$ be a continuous signature, and let $\mathcal{M}$, $\mathcal{N}$, $\mathcal{U}$ be $\Lan$-pre-structures.
        \begin{enumerate}
       \item[(i)] We say that $\mathcal{N}$ is a \emph{one-point extension} of $\mathcal{M}$ if $\mathcal{M}$ is a substructure of $\mathcal{N}$ and $N\setminus M$ is a singleton.
       \item[(ii)] We say that $\mathcal{U}$ has the \emph{Urysohn property} if given any finite continuous $\Lan$-structure $\mathcal{M}$, a one-point extension $\mathcal{N}$ of $\mathcal{M}$, and an isomorphic embedding $\varphi$ from $\mathcal{M}$ into $\mathcal{U}$, there is an isomorphic embedding $\psi$ from $\mathcal{N}$ to $\mathcal{U}$ such that $\psi\upharpoonright M=\varphi$. We say $\mathcal{U}$ is \emph{Urysohn} if it satisfies the Urysohn property.
        \end{enumerate}
    \end{definition}
    
\begin{definition}[{\cite[Definitions 3.2 and 4.11]{Gao-Ren}}]
    Let $\Lan$ be a continuous signature with only finitely many relation symbols.
    \begin{enumerate}
    \item[(i)] We say that $\Lan$ is \emph{semiproper} if the following hold:
    \begin{itemize}
    \item For any $n\sh$ary $R\in\Lan$ and $1\leq i\leq n$, let $u_{R,i}$ be the modulus function for the $i$-th coordinate. Then
    \begin{align*}
        I_{R,i}=\{r\in[0,+\infty]:u_{R,i}(r)<1\}
    \end{align*}
    is bounded, and $u_{R,i}$ is \emph{superadditive} on $I_{R,i}$, i.e., for any $r_1, r_2\in I_{R,i}$, 
    $$ u_{R,i}(r_1+r_2)\geq u_{R,i}(r_1)+u_{R,i}(r_2); $$ and
   \item for any $n\sh$ary $R\in\Lan$, where $n\geq 2$, and $1\leq i\leq n,$ there is $K_{R,i}>0$ such that $u_{R,i}(r)=K_{R,i}r$ for all $r\in I_{R,i}$.
    \end{itemize}
    \item[(ii)] $\Lan$ is \emph{proper} is $\Lan$ is semiproper and each $u_{R,i}$ is upper semicontinuous on $I_{R,i}$.
    \item[(iii)] $\Lan$ is \emph{Lipschitz} if $\Lan$ is proper and for each unary $R\in\Lan$ there is $K_{R,1}>0$ such that $u_{R,1}(r)=K_{R,1}r$ for all $r\in I_{R,1}$.
    \end{enumerate}
\end{definition}

One of the main results of \cite{Gao-Ren} is the following characterization of the existence of Urysohn continuous structure.

    \begin{theorem}[{\cite[Theorem 1.1]{Gao-Ren}}]
        Let $\Lan$ be a continuous signature with only finitely many relation symbols. Then there exists a separable Urysohn $\Lan$-structure if and only if $\Lan$ is proper.
    \end{theorem}
   
For a proper continuous signature $\Lan$, there exists a unique separable $\Lan$-structure with the Urysohn property. We denote it by $\U_\Lan$ and call it the \emph{Urysohn (continuous) $\Lan$-structure}. %
        In particular, for every Lipschitz continuous signature $\Lan$ the Urysohn continuous structure $\U_\Lan$ exists.
We use $\Iso(\U_\Lan)$ to denote the group of all automorphisms of the continuous $\Lan$-structure $\U_\Lan$ equipped with the pointwise convergence topology, where the underlying space $\U_\Lan$ is equipped with the metric topology. 
    
The following result is also of interest.
\begin{theorem}[{\cite[Theorem 4.12]{Gao-Ren}}]
Let $\Lan$ be a proper continuous structure. Then $\Lan$ is Lipschitz if and only if the  underlying metric space  $\U_\Lan$ is isometric to $\U$.
\end{theorem}

For any proper $\Lan$, the  underlying metric space  $\U_\Lan$ is an uncountable Polish metric space. It is more convenient for us to work with a dense countable pre-structure in $\U_\Lan$. The following concept gives a nice notion of such countable pre-structure. %

\begin{definition}[{\cite[Definition 5.1]{Gao-Ren}}]
    Let $\Lan$ be a continuous signature, let $\Delta$ be a distance value set, and let $V\subseteq[0,1]$. We say that $(\Delta,V)$ is a \emph{good value pair} for $\Lan$ if:
    \begin{itemize}
        \item whenever $\Delta$ is bounded, then for any $n\sh$ary $R\in\Lan$ and $1\leq i\leq n,\ u_{R,i}(\sup \Delta)\geq 1$,
        \item $0\in V$,
        \item for any $v\in V$, any $\delta\in\Delta$, any $n\sh$ary $R\in\Lan$ and $1\leq i\leq n$, if $v>u_{R,i}(\delta)$, then $v-u_{R,i}(\delta)\in V$.
    \end{itemize}
In addition, we say that a good value pair $(\Delta,V)$ is \emph{finite} (\emph{countable}, respectively) if both $\Delta$ and $V$ are finite (countable, respectively).
\end{definition}

\begin{definition}[{\cite[Definition 5.2]{Gao-Ren}}]
    Let $\Lan$ be a continuous signature and $\M$ a continuous $\Lan\sh$pre-structure. Let $(\Delta,V)$ be a good value pair for $\Lan$.
    \begin{enumerate}
        \item[(i)] We say that $\M$ is $(\Delta,V)$-\emph{valued} if for all $x,y\in M,d_M(x,y)\in\Delta$ and for all $n\sh$ary $R\in\Lan$ and $\Bar{x}\in M^n,R^M(\Bar{x})\in V$.
        \item[(ii)] Let $\K_{(\Delta,V)}$ be the class of all finite $(\Delta,V)$-valued $\Lan\sh$structures.
    \end{enumerate}
\end{definition}
\begin{lemma}[{\cite[Lemma 5.3]{Gao-Ren}}]\label{gdpair}
    Let $\Lan$ be a semiproper continuous signature.
    \begin{enumerate}
        \item For any finite $\Lan$-structure $\M$, there exists a finite good value pair $(\Delta,V)$ for $\Lan$ such that $\M$ is $(\Delta,V)$-valued.
        \item For any finite (countable, respectively) distance value set $\Delta$ and finite (countable, respectively) $W\subseteq[0,1]$, there exists a finite (countable, respectively) $V$ such that $W\subseteq V$ and $(\Delta,V)$ is a good value pair for $\Lan$.
    \end{enumerate}
\end{lemma}

Here we make a new definition to simplify our proofs in the rest of this section.

\begin{definition} Let $\Lan$ be a continuous signature and let $(\Delta, V)$ be a good value pair for $\Lan$. We say that $(\Delta, V)$ is \emph{normal} if the following hold:
    \begin{itemize}
        \item $\Delta=\mathbb{Q}^+$;
        \item $V$ is countable and $V$ is closed under taking arithmetic means, i.e., for any $n\in\N^+$ and $v_1, \dots, v_n\in V$, we have
        $$ \displaystyle\frac{1}{n}\sum_{i=1}^n v_i \in V. $$
    \end{itemize}
    \end{definition}
Note that normal good value pairs are necessarily countable, and they exist for semiproper $\Lan$ by a similar argument as in the proof of \cite[Lemma 5.3]{Gao-Ren}.

We will also use the following theorem.
 \begin{theorem}[{\cite[Theorem 5.5]{Gao-Ren}}]
     Let $\Lan$ be a semiproper continuous signature and let $(\Delta,V)$ be a countable good value pair for $\Lan$. Then $\K_{(\Delta,V)}$ is a countable Fra\"iss\'e class.
 \end{theorem}
 
 By the standard Fra\"iss\'e theory, under the assumptions of the above theorem there exists a Fra\"iss\'e limit of $\K_{(\Delta,V)}$, which we denote by $\U_{(\Delta,V)}$. In particular, for a normal good value pair $(\Delta,V)$, there exists a countable pre-structure $\U_{(\Delta,V)}$ which is dense in $\U_\Lan$. Let $\Iso(\U_{(\Delta, V)})$ denote the group of all automorphisms of $\U_{(\Delta, V)}$ equipped with the pointwise convergence topology, where the underlying space $\U_{(\Delta, V)}$ is equipped with the metric topology. Then $\Iso(\U_{(\Q,V)})$ is separable metrizable, and is dense in $\Iso(\U_\Lan)$. 
 
 Our main goal in this section is to show that if $\Lan$ is Lipschitz and $(\Delta, V)$ is a normal good value pair for $\Lan$, then $\Iso(\U_{(\Delta,V)})$ has the strong L\'evy property. This implies that the Polish group $\Iso(\U_\Lan)$ has the strong L\'evy property under the same assumptions.

Similar to our approach of Section~\ref{first:density:section}, we consider the class of isomorphic actions of finite groups on finite $\Lan$-structures. This was already defined and studied in \cite{Gao-Ren}, as follows. %

\begin{theorem}[{\cite[Theorem 8.1, Lemmas 8.2--8.5]{Gao-Ren}}]
    Let $\Lan$ be a semiproper continuous signature and let $(\Delta,V)$ be a countable good value pair for $\Lan$. Let $\C_{(\Delta,V)}$ be the class consisting of actions by automorphisms $G\curvearrowright\M$, where $G$ is a finite group and 
$\M\in\K_{(\Delta,V)}$. Then $\C_{(\Delta,V)}$ is a countable Fra\"iss\'e class. Moreover, if we denote the Fra\"iss\'e limit of $\C_{(\Delta,V)}$ by $H_{(\Delta, V)}\curvearrowright\M_{(\Delta, V)}$, then we have
    \begin{itemize}
        \item $\M_{(\Delta, V)}$ is isomorphic to $\U_{(\Delta,V)}$;
        \item $H_{(\Delta, V)}$ is isomorphic to $\H$ and acts faithfully on $\M_{(\Delta, V)}$;
        \item  $H_{(\Delta, V)}$ is dense in $\Iso(\M_{(\Delta,V)})$.
    \end{itemize}
\end{theorem}

The following is our main theorem of this section. 
\begin{theorem}\label{this:main-cont:thm}
Let $\Lan$ be a Lipschitz continuous signature and let $(\Delta,V)$ be a normal good value pair. Then $H_{(\Delta, V)}$ equipped with the subspace topology of $\Iso(\U_{(\Delta, V)})$ has the strong L\'evy property.
\end{theorem}

The main theorem has the following immediate corollary.

\begin{corollary}\label{cor:cont} Let $\Lan$ be any Lipschitz continuous signature. Then $\Iso(\U_\Lan)$ has the strong L\'evy property.
\end{corollary}

The rest of this section is devoted to a proof of Theorem~\ref{this:main-cont:thm}. The key ingredient of the proof is the following analog of Lemma~\ref{this:density:lemma} in the context of continuous logic. 

\begin{lemma} \label{this:density-cont:lemma}
Let $\Lan$ be a Lipschitz continuous signature. Let $(\Delta,V)$ be a normal good value pair for $\Lan$. For each $n \in \N^+$, the set \begin{equation}
\mathcal{D}_{n} = \{(\M,G) \in \mathcal{C}_{(\Delta,V)}\colon  \alpha_{G} \leq \beta_n \}.
\end{equation}
is  cofinal  in $\mathcal{C}_{(\Delta,V)}$.
\end{lemma}
\begin{proof}
    Suppose $(G\curvearrowright\M)\in \mathcal{C}_{(\Delta, V)}$ and $\M=(M,d_M,\{R_i^M\}_{i\leq k})$. Let $n\in\N^+$ be fixed. Let $m\in\N^+$ be a positive integer such that $m\geq 16a^2n$. 
We define an $\Lan$-structure $\mathcal{N}=(N, d_N, \{R_i^N\}_{i\leq k})$, where $N=M^m$ and $d_N$ is the normalized Hamming metric on $N$. For every $i\leq k$, assuming that $R_i$ is $r_i$-ary, define $R_i^N\colon (M^{m})^{r_i}\to V$ by     \begin{align*}
R_i^N(\Bar{x}_1,\Bar{x}_2,...,\Bar{x}_{r_i})=\frac{1}{m}\sum_{l=1}^mR_i^M(x^l_1,x^l_2,...,x^l_{r_i})
    \end{align*}
    for $\Bar{x}_j=(x^1_j,x^2_j,...,x^{m}_j)$, $1\leq j\leq r_i$.

Fix $1\leq j\leq r_i$ and $K_j>0$. Suppose $u_j(r)=K_jr$ is the modulus function for the $j$-th coordinate of $R_i^M$. If $\Bar{x}_1, \dots, \Bar{x}_{r_i}\in M^m$ and $\Bar{y}_j\in M^m$, then
$$ \begin{array}{rl} & |R_i^N(\dots, \Bar{x}_{j-1}, \Bar{x}_j, \Bar{x}_{j+1},\dots)-
R_i^N(\dots, \Bar{x}_{j-1}, \Bar{y}_j, \Bar{x}_{j+1},\dots)| \\
\leq & \displaystyle\frac{1}{m}\sum_{l=1}^m |R_i^M(\dots, x_{j-1}^l, x_j^l, x_{j+1}^l,\dots)-R_i^M(\dots, x_{j-1}^l, y_j^l, x_{j+1}^l,\dots)| \\
\leq & \displaystyle\frac{1}{m}\sum_{l=1}^m K_j d_M(x_j^l, y_j^l) =K_j d_N(\Bar{x}_j, \Bar{y}_j). 
\end{array}
$$ 
Thus $\mathcal{N}$ is an $\mathcal{L}$-structure. Since $\Delta=\Q^+$, it is easily seen that if $(M, d_M)$ is a $\Delta$-metric space then so is $(N, d_N)$. By the normality of $(\Delta, V)$, if $\mathcal{M}$ is $(\Delta, V)$-valued then so is $\mathcal{N}$. 
   
Let $\pi\colon M\to N$ be given by 
$$ \pi(x)=(x,\dots, x)\in M^m $$
for any $x\in M$. Let $\phi\colon G\to G^m$ be similarly defined. Then $\pi$ is an isometric embedding from $(M, d_M)$ to $(N, d_N)$, and $\phi$ is a group embedding from $G$ to $G^m$. Again consider the action of $G^m$ on $\mathcal{N}$ defined by
$$ \Bar{g}\cdot \Bar{x}=(g^1\cdot x^1, \dots, g^m\cdot x^m)), $$ 
for $\Bar{g}=(g^1,\dots, g^m)\in G^m$ and $\Bar{x}=(x^1,\dots, x^m)\in M^m$.  
It is easily seen that the action preserves the metric $d_N$. Next, we check that it also preserves the relations $\{R_i^N\}_{i\leq k}$. Let $\Bar{g}\in G^{m}$ and $\Bar{x}_1,\dots, \Bar{x}_{r_i}\in M^m$. Then for each $i\leq k$, we have
    \begin{align*}
        R_i^N(\Bar{g}\cdot\Bar{x}_1,\Bar{g}\cdot\Bar{x}_2,...,\Bar{g}\cdot\Bar{x}_{r_i})&=\frac{1}{m_n}\sum_{l=1}^m R_i^M(g^l\cdot x^l_1,g^l\cdot x^l_2,...,g^l\cdot x^l_{r_i})\\
        &=\frac{1}{m}\sum_{l=1}^m R_i^M(x^l_1, x^l_2,...,x^l_{r_i})\\
        &=R_i^N(\Bar{x}_1,\Bar{x}_2,...,\Bar{x}_{r_i}).
    \end{align*}
Similarly, it is straightforward to check that the embedding $(\phi, \pi)$ preserves not only the metric and the action, but also the relations in the structures.
 
 Finally, note that in defining $\mathcal{N}$ we did not do a remetrization, and therefore Lemma~\ref{this:lemma:length} and Theorem~\ref{this:concentration:thm} can be applied with arbitrary $\epsilon>0$. This implies
 $$ \alpha_{G^m}(\varepsilon)\leq 2e^{\frac{-m\varepsilon^2}{16a^2}}\leq 2e^{-n\varepsilon^2}=\beta_n(\varepsilon) $$
 for any $\varepsilon>0$. This shows that $(G^m\curvearrowright \mathcal{N})\in \mathcal{D}_n$.
\end{proof}

We are now ready to prove Theorem \ref{this:main-cont:thm}.
\begin{proof}[Proof of Theorem \ref{this:main-cont:thm}]
The proof is identical to that of Theorem~\ref{thm:UD1}, except that Lemma \ref{this:density:lemma} is now replaced by the analogous Lemma \ref{this:density-cont:lemma} for continuous logic. 
\end{proof}

\section{The Isomorphism Types of Strong L\'evy Groups}\label{section:complexity}
In this final section we draw some conclusions about the isomorphism types of isometry groups and countable groups with the strong L\'evy property. 

We first focus  on  countable groups. The countable groups we considered in this paper are all infinite and metrizable. Thus we can use the following coding space for all of them. Let $\mathfrak{G}$ be the subset of
$$ X=2^{\N^2}\times \R^{\N^2} $$
consisting of all pairs $(R, d)$ where $(\N, R)$ is a group and $d$ is a metric on $\N$. Equipped with the product topology, $X$ is a Polish space, and $\mathfrak{G}$ is a Borel subset. Hence $\mathfrak{G}$ is a standard Borel space (see, e.g., Kechris \cite[Corollary 13.4]{KechrisBook}). We view elements of $\mathfrak{G}$ to be codes for countable metrizable groups. For each $\gamma\in \mathfrak{G}$, we let $G_\gamma$ denote the countable metrizable group coded by $\gamma$. For $\gamma, \eta\in \mathfrak{G}$, write $\gamma\cong_{\mathfrak{G}}\eta$ if $G_\gamma$ and $G_\eta$ are isomorphic as topological groups. Then $\cong_{\mathfrak{G}}$ is an equivalence relation on $\mathfrak{G}$.

In descriptive set theory, we use the notion of Borel reducibility to compare the relative complexity of equivalence relations on standard Borel spaces (see, e.g., \cite[Chapter 5]{GaoBook}). We recall this notion below. If $X, Y$ are standard Borel spaces and $E, F$ are equivalence relations on $X, Y$, respectively, then we say that $E$ is \emph{Borel reducible} to $F$, denoted $E\leq_B F$, if there is a Borel function $f\colon X\to Y$ such that for all $x_1, x_2\in X$,
$$ x_1Ex_2\iff f(x_1)Ff(x_2). $$
In this case we  intuitively  regard $E$ to be no more complex than $F$. We say that $E$ is \emph{strictly Borel reducible} to $F$, denoted $E<_B F$, if $E\leq_B F$ but $F\not\leq_B E$. In this case we intuitively regard $E$ to be strictly simpler than $F$.

The following benchmark equivalence relations are very useful in distinguishing between the complexities of equivalence relations. 

\begin{example}{\ }
\begin{enumerate}
\item The identity or equality equivalence relation on $\R$, i.e., for $x, y\in \R$, $x$ and $y$ are equivalent if and only if $x=y$. We use $=_\R$ or $\mbox{id}_\R$ to denote this equivalence relation. 
\item The equivalence relation $=^+$ on $\R^\N$, defined by
$$ (x_n)_n =^+ (y_n)_n\iff \{x_n\colon n\in \N\}=\{y_n\colon n\in\N\}. $$
\item The \emph{graph isomorphism} relation. Let $\mathfrak{R}$ be the subset of $2^{\N^2}$ consisting of all binary relations $R$ where $(\N, R)$ is a graph. Then $\mathfrak{R}$ is a Borel subset of the Polish space $2^{\N^2}$, and hence is a standard Borel space. For each $\alpha\in \mathfrak{R}$, let $\Gamma_\alpha$ be the countable graph coded by $\alpha$. For $\alpha, \beta\in\mathfrak{R}$, write $\alpha\cong_{\mathfrak{R}}\beta$ if $\Gamma_\alpha$ and $\Gamma_\beta$ are isomorphic as graphs. Then $\cong_{\mathfrak{R}}$ is known as the graph isomorphism relation.
\end{enumerate}
\end{example}

All of these equivalence relations have been well studied, and in particular it is well known that
$$ =_\R\ <_B\ =^+\ <_B\ \cong_{\mathfrak{R}}. $$
Statements about Borel reducibility are strengthenings of statements about cardinalities. For example, if $=_\R\ \leq_B E$, then it follows that there are continuum many distinct $E$-classes.

It was shown in \cite[Theorem 5.5]{Gao-Adv} that there are continuum many pairwise nonisomorphic countable omnigenous locally finite groups. Restated in the language of Borel reducibility, it is essentially saying that $=_\R$ is Borel reducible to the isomorphism relation for countable omnigenous locally finite groups. In the recent paper \cite{Gao-Li}, Gao and Li showed that the graph isomorphism relation is Borel reducible to the isomorphism relation for countable omnigenous locally finite groups (\cite[Theorem 1.1]{Gao-Li}). Combined with our results in Section~\ref{section:omnigen:construction}, we obtain the following corollary.

\begin{corollary} The graph isomorphism relation is Borel reducible to the isomorphism relation for countable metrizable groups with the L\'evy property. In other words, there is a Borel map $f\colon \mathfrak{R}\to \mathfrak{G}$ such that 
\begin{enumerate}
\item[(i)] for any $\alpha\in\mathfrak{R}$, $G_{f(\alpha)}$ has the L\'evy property;
\item[(ii)] for any $\alpha, \beta\in \mathfrak{R}$, we have $\alpha\cong_{\mathfrak{R}}\beta$ if and only if $f(\alpha)\cong_{\mathfrak{G}}f(\beta)$.
\end{enumerate}
\end{corollary}

In some sense, this is saying that there are significantly more isomorphism types of countable L\'evy groups than just continuum many. 

Next we turn to isometry groups. The isomorphism types of groups of the form $\Iso(\U_\Delta)$, where $\Delta$ is a countable distance value set, have been studied by Etedadialiabadi, Gao and Li in \cite{Eted-Gao-Feng}. The following result completely characterizes their isomorphism types in terms of the distance value sets.

\begin{theorem}[{\cite[Theorem 1.2]{Eted-Gao-Feng}}]
    Let $\Delta$ and $\Lambda$ be countable distance value sets. Then
the following are equivalent:
\begin{enumerate}
\item[(1)] $\Iso(\U_\Delta)$ and $\Iso(\U_\Lambda)$ are isomorphic as topological groups.
\item[(2)] There exists a bijection $f\colon \Delta \to \Lambda$ such that for any $x, y, z\in \Delta$, $(x,y,z)$ is a $\Delta$-triangle if and only if $(f(x), f(y), f(z))$ is a $\Lambda$-triangle.
\end{enumerate}
\end{theorem}

Clause (2) defines an equivalence relation between countable distance value sets, which we denote by $\sim_{\mathsf{CDVS}}$. In \cite{Eted-Gao-Feng} a coding space $\mathsf{CDVS}$ for all countable distance value sets was introduced, which is a standard Borel space. $\sim_{\mathsf{CDVS}}$ is an equivalence relation on $\mathsf{CDVS}$. Let $\mathsf{CDVS}_0$ be the subspace of $\mathsf{CDVS}$ consisting of codes for all countable distance value sets $\Delta$ with $\inf\Delta=0$. Then $\mathsf{CDVS}_0$ is also a standard Borel space. Let $\sim_0$ be the equivalence relation $\sim_{\mathsf{CDVS}}$ restricted on $\mathsf{CDVS}_0$. It follows from \cite[Lemma 6.4 and Proposition 6.5]{Eted-Gao-Feng} that
$$=^+\ <_B\ \sim_0\ <_B\ \cong_{\mathfrak{R}}. $$
Combined with our Corollary~\ref{cor:main}, we have the following corollary.

\begin{corollary} There is an assignment $x\mapsto G_x$ which associates to any $x\in \R^{\N}$ a separable metrizable group $G_x$ with the strong L\'evy property such that for any $x, y\in \R^{\N}$, we have $x=^+y$ if and only if $G_x$ and $G_y$ are isomorphic as topological groups.
\end{corollary}

Indeed, we know that the isomorphism relation between separable metrizable groups is strictly more complex than $=^+$.

\end{document}